\documentclass{article}[12pt]
\usepackage{url, graphicx, xcolor, hyperref, geometry}
\usepackage{latexsym,amsbsy}
\usepackage{lipsum}
\usepackage{amsfonts}
\usepackage{graphicx}
\usepackage{epstopdf}
\usepackage{algorithm}
\usepackage{bm}
\usepackage{amssymb,amsmath,amsthm}
\usepackage{smartdiagram}
\usepackage{mathabx}
\usepackage{float}
\usepackage{enumerate}
\usepackage{diagbox}
\usepackage{algpseudocode}
\newtheorem{theorem}{Theorem}[section]
\newtheorem{lemma}{Lemma}[section]

\newtheorem{remark}{Remark}[section]

 \newtheorem{example}{Example}[section]

\hypersetup{
    colorlinks,
    linkcolor={blue},
    citecolor={blue},
    urlcolor={blue}
}

\begin{document}
 
\title{Convergence Analysis of An Alternating Nonlinear GMRES on Linear Systems}

\author{ Yunhui He\thanks{Department of Mathematics, University of Houston,  3551 Cullen Blvd, Houston, Texas 77204-3008, USA.(\tt{yhe43@central.uh.edu}).}}
\date{\today}
\maketitle

\begin{abstract}
In this work, we develop an alternating nonlinear Generalized Minimum Residual (NGMRES) algorithm with depth $m$ and periodicity $p$, denoted by aNGMRES($m, p$), applied to linear systems. We provide a theoretical analysis to quantify by how much one-step NGMRES($m$) using Richardson iterations as initial guesses can improve the convergence speed of the underlying fixed-point iteration for diagonalizable and symmetric positive definite cases. Our theoretical analysis gives us a better understanding of which factors affect the convergence speed. Moreover, under certain conditions, we prove the periodic equivalence between the proposed aNGMRES applied to Richardson iteration and GMRES. Specifically, aNGMRES($\infty,p$) and full GMRES are identical at the iteration index $jp$. Therefore, aNGMRES($\infty,p$) can be regarded as an alternative to GMRES for solving linear systems. For finite $m$, the iterates of aNGMRES($m,m+1$) and restarted GMRES (GMRES($m+1$)) are the same at the end of each periodic interval of length $p$, i.e, at the iteration index $jp$.  In Addition, we present a convergence analysis of aNGMRES when applied to accelerate Richardson iteration. The advantages of aNGMRES($m,p$) method are that there is no need to solve a least-squares problem at each iteration which can reduce the computational cost, and it can enhance the robustness against stagnations, which could occur for NGMRES($m$).
\end{abstract}


\section{Introduction}\label{sec:intro}
 	We consider the linear system
 	\begin{equation}\label{eq:linear-sys}
 		Au=b,
 	\end{equation}
 	where $A\in \mathbb{R}^{n\times n}$ is invertible, and $u, b\in\mathbb{R}^n$. Assume that $u^*$ is the exact solution. There are various iterative methods to solve \eqref{eq:linear-sys}, such as Jacobi iteration, Gauss-Seidel, successive over-relaxation, Krylov subspace methods, preconditioned iterations, and multigrid methods.  Acceleration methods, such as Chebyshev acceleration \cite{d2021acceleration},  Nesterov acceleration \cite{nesterov2013introductory}, Anderson-type acceleration \cite{anderson1965iterative,walker2011anderson,evans2020proof,chen2022composite,pollock2019anderson,anderson2019comments}, and nonlinear GMRES (NGMRES) \cite{washio1997krylov,oosterlee2000krylov}, are often used to speed up the convergence of iterative methods. There is increasing interest in understanding and evaluating how acceleration methods can enhance the underlying iterative methods and their applications.  In this work, we are interested in using NGMRES to accelerate fixed-point iterations applied to \eqref{eq:linear-sys}, and in analyzing its convergence properties as well as its connections to existing methods. NGMRES was originally proposed by Washio and Oosterlee \cite{washio1997krylov} to accelerate nonlinear multigrid methods.   We note that there are few studies on the development of NGMRES. In De Sterck's work \cite{sterck2013steepest}, steepest descent preconditioning for NGMRES was examined, which significantly accelerates the convergence of stand-alone steepest descent optimization. While NGMRES has demonstrated competitive performance compared to other acceleration
 	methods, its theoretical foundations, algorithmic improvements, and potential applications remain
 	underexplored---even in the content of linear problems. Our goal is to better understand NGMRES applied to linear systems, and to propose a new variant of NGMRES. 
 	
 	First, we provide a brief review of NGMRES.  NGMRES with a finite window size $m$ (i.e., using the previous $m +1$ iterates), denoted by NGMRES($m$), is preferred over full version because of its greater computational efficiency. In 1997,  Washio and Oosterlee \cite{washio1997krylov} proposed Krylov subspace acceleration analogous to GMRES for the nonlinear multigrid on the finest level. In 2000,  Oosterlee and Washio \cite{oosterlee2000krylov} proposed a parallel nonlinear Krylov acceleration strategy for solving nonlinear equations, which utilized multigrid and GMRES methods. Later, NGMRES \cite{sterck2012nonlinear,sterck2021asymptotic} is applied to accelerate the alternating least squares for canonical tensor decomposition problems.  However, there is little mathematical understanding of the theory or solution approaches for nonlinear problems with NGMRES. Analysis of the asymptotic convergence of NGMRES can provide a characterization of the behavior of NGMRES. There are fundamental gaps in our current understanding of how NGMRES works, and what the relationships between NGMRES, GMRES \cite{saad2003iterative}, and restarted GMRES (GMRES($m$)) \cite{saad1986gmres} are. The worst-case root-convergence of GMRES(1)  has been proposed \cite{He25worstcase}. In fact, NGMRES(0) is GMRES(1). Thus, the results of GMRES(1) \cite{He25worstcase} serve as an upper bound on the performance of NGMRES($m$) on linear systems. Recently, we built a connection between NGMRES and GMRES when applied to linear systems. Assume that the 2-norm of the residual of GMRES is strictly decreasing. We \cite{GreifHe25NGMRES} proved that the residuals of NGMRES($\infty$) are the same as these of GMRES, and  if $A$ is invertible, the iterates generated by the two methods are identical. Moreover,  for invertible matrices $A$ that are either symmetric or shifted skew-symmetric of the form $A=\alpha I + S$, where $\alpha \in \mathbb{R}$ and $S$ is skew-symmetric, we proved that NGMRES($m$) and GMRES are identical for $\forall m\geq 1$, suggesting that a large window size $m$ is unnecessary. Specially, when $A$ is symmetric and invertible, NGMRES(1) is equivalent to the conjugate residual method. These discoveries help us better understand NGMRES on linear problems. For nonlinear problems, we have provided convergence analysis for NGMRES($m$) for a specific class of contractive fixed-point operators\cite{He25NGMRES}. Specifically, we proved that for general $m > 0$, the residuals of NGMRES($m$) converge r-linearly under some conditions. For $m = 0$, we proved that the residuals of NGMRES(0) converge q-linearly.  In practice, we see NGMRES dramatically enhances the performance of linear and  nonlinear fixed-point iterations.  However, there is a lack of studies that quantitatively assess the improvements achieved by NGMRES compared to the underlying fixed-point iteration. One aim of this work is to quantify the improvement achieved by considering a single step of NGMRES  applied to linear systems using Richardson iterations.
 	
 	In this work, we first study how NGMRES can improve the performance of a fixed-point iteration by considering one-step NGMRES on linear problems, that is, given $u_0$, we compute $u_j=q(u_{j-1})$, where $j=1,2,\cdots, k$, and then generate $u_{k+1}$ by NGMRES($m$). Note that at each iteration of NGMRES($m$), one needs to solve a least-squares problem, which might be computational expensive for large $m$. Then, based on one-step analysis of NGMRES, we propose alternating NGMRES with depth $m$ and periodicity $p$, denoted by aNGMRES($m,p$). The main idea is that, given an initial guess $x_0$, we compute $p-1$   fixed point iterations, then switch to a single NGMRES($m$) iteration. This process is then repeated. We provide convergence analysis of aNGMRES, and make connections between aNGMRES, GMRES, and restarted GMRES. Additionally, we make comments on recently proposed alternating Anderson-Richardson method (AAR). Our proposed aNGMRES method offers a novel approach to solving large linear systems. It can be extended to accelerate nonlinear fixed-point iterations.

 	The remainder of this work is organized as follows. In section \ref{sec:NGMRES}, we introduce NGMRES and analyze the improvement of one- step NGMRES in two distinct cases: (1) $M=I-A$ is diagonalizable, and (2) when $M$ is symmetric and positive definite. In section \ref{sec:aNGMERS}, we propose aNGMRES, establish relationships with full GMRES and restarted GMRES, and provide convergence analysis of aNGMRES. In section \ref{sec:num}, we present some numerical examples to validate our theoretical findings. Additionally, we discuss the practical considerations of aNGMRES and present the performance of aNGMRES applied to both linear and nonlinear problem settings. We draw conclusions in section \ref{sec:con}.

 	\section{NGMRES}\label{sec:NGMRES}
 	NGMRES \cite{washio1997krylov, oosterlee2000krylov} was originally developed to accelerate nonlinear iterations. Here, we consider NGMRES applied to linear systems. Assume we have the following fixed-point iteration for \eqref{eq:linear-sys}
 	\begin{equation}\label{eq:FP}
 		u_{k+1}=q(u_k)=Mu_k+b,
 	\end{equation}
 	where $M=I-A$. Note that  \eqref{eq:FP} is the Richardson iteration.  Define the $k$th residual as
 	\begin{equation}\label{eq:rk-u-minus-q}
 		r_k=r(u_k)=u_k-q(u_k)=Au_k-b.
 	\end{equation}
 	In practice, the fixed-point iteration \eqref{eq:FP} converges slowly or even diverges. We consider NGMRES($m$)  to accelerate \eqref{eq:FP}. In each iteration,  a least-squares problem must be solved to determine the combination coefficients. We present the framework of NGMRES($m$) in Algorithm \ref{alg:NGMRESm}. Throughout this work,  $\|\cdot\|$ denotes the 2-norm ($\ell^2$) of a vector or a matrix.
 	
 	\begin{algorithm}
 		\caption{A windowed NGMRES with depth $m$: NGMRES($m$)} \label{alg:NGMRESm}
 		\begin{algorithmic}[1] 
 			\State Given $u_0$ and $m\geq 0$ 
 			\For {$k=0,1,2,\cdots$}
 			\begin{itemize}
 				\item set $m_k=\min\{k,m\}$;
 				\item compute the residual  $r_k=u_k-q(u_k)$; 
 				\item compute 
 				\begin{equation}\label{eq:ngukp1}
 					u_{k+1}=q(u_k)+\sum_{i=0}^{m_k} \beta^{(k)}_i\left(q(u_k)- u_{k-i})\right),
 				\end{equation}
 				where ${\bm\beta^{(k)}}=[\beta_0^{(k)},\beta_1^{(k)},\cdots, \beta_{m_k}^{(k)}]^T$ is obtained by solving the following least-squares problem
 				\begin{equation}\label{eq:NGMRES-LSQ}
 					\min_{\bm \beta^{(k)}}\|r(q(u_k))+\sum_{i=0}^{m_k} \beta^{(k)}_i\left(r(q(u_k))-r(u_{k-i})\right)\|^2.
 				\end{equation}
 			\end{itemize}
 			\EndFor
 		\end{algorithmic}
 	\end{algorithm}
 	
 	\begin{remark}
 			We note that, when applying NGMRES to solve the general equation $g(u)=0$, the corresponding least-squares problem in Algorithm \ref{alg:NGMRESm} takes the form\cite{GreifHe25NGMRES}
 			\begin{equation*} 
 				\min_{\bm \beta^{(k)}}\|g(q(u_k))+\sum_{i=0}^{m_k} \beta^{(k)}_i\left(g(q(u_k))-g(u_{k-i})\right)\|^2.
 			\end{equation*}
 			In this work, we focus on solving $Au=b$, which can be reformulated as $g(u)=Au-b=0$, using the fixed-point iteration defined in \eqref{eq:FP}. From \eqref{eq:rk-u-minus-q}, we observe that $r(u_k)=g(u_k)$, which justifies our use of the least-squares formulation in \eqref{eq:NGMRES-LSQ} throughout this work.
 		\end{remark}

The work of Greif and He\cite{GreifHe25NGMRES} shows that the least-squares problem \eqref{eq:NGMRES-LSQ} is minimizing the $(k+1)$th residual. For completeness, we restate their key results as a lemma below.
 	\begin{lemma}\cite{GreifHe25NGMRES} \label{lem:NGMRES-mini-resi}
 		Considering NGMRES($m$) in  Algorithm \ref{alg:NGMRESm} for $q(u)=Mu+b$, we have
 		\begin{equation}\label{eq:NG-minrk}
 			\|r_{k+1}\|=\|Au_{k+1}-b\|=	\min_{\bm\beta^{(k)}}\|r(q(u_k))+\sum_{i=0}^{m_k} \beta^{(k)}_i\left(r(q(u_k))-r(u_{k-i})\right)\|^2,
 		\end{equation}
 		and
 		\begin{equation*}\label{eq:NGMRES-nonincre}
 			\|r_{k+1}\|\leq \|r_k\|.
 		\end{equation*}
 	\end{lemma}
 	We omit the proof and refer the reader to the work of Greif and He\cite{GreifHe25NGMRES} for further details.  Since the 2-norm of the residuals of NGMRES($m$) is nonincreasing, NGMRES($m$) always converges. 
 	Let $\gamma_j^{(k)}=\sum_{i=j}^{m_k}\beta_i^{(k)}$ and $\bm{\gamma}^{(k)}=[\gamma_{m_k}^{(k)}, \gamma_{m_k-1}^{(k)},\cdots,\gamma_{0}^{(k)}]^T$.  For the convenience of convergence analysis, we rewrite Algorithm \ref{alg:NGMRESm} as Algorithm  \ref{alg:NGMRESm-simp}.

 	\begin{algorithm}[h]
 		\caption{Alternative format of windowed NGMRES with depth $m$: NGMRES($m$)} \label{alg:NGMRESm-simp}
 		\begin{algorithmic}[1] 
 			\State Given $u_0$ and $m\geq 0$ 
 			\For {$k=0,1,2,\cdots$}
 			\begin{itemize}
 				\item set $m_k=\min\{k,m\}$;
 				\item compute the residual  $r_k=u_k-q(u_k)$; 
 				\item compute 
 				\begin{equation}\label{NGMRES-update}
 					u_{k+1}=q(u_k)+\gamma_0^{(k)}(q(u_k)-u_k) + \sum_{i=1}^{m_k} \gamma^{(k)}_i\left(u_{k-i+1}- u_{k-i}\right),
 				\end{equation}
 				where  $\gamma_i^{(k)}$ is obtained by solving the following least-squares problem
 				\begin{equation}\label{eq:NGMRES-LSQ-simp}
 					\min_{\bm\gamma^{(k)}}\|r(q(u_k))+\gamma_0^{(k)}(r(q(u_k))-r(u_k)) + \sum_{i=1}^{m_k} \gamma^{(k)}_i\left(r(u_{k-i+1})- r(u_{k-i})\right)\|^2.
 				\end{equation}
 			\end{itemize}
 			\EndFor
 		\end{algorithmic}
 	\end{algorithm}
 	
 To carry out the convergence analysis of NGMRES, our first objective is to establish the relationship between consecutive residuals or errors. Now we define the following matrices with dimensions $n\times (m_k+1)$, 
 	\begin{equation*}
 		\mathcal{S}_k=
 		\begin{bmatrix}
 			u_{k-m_k+1}-u_{k-m_k},& u_{k-m_k+2}-u_{k-m_k+1}, &\cdots & u_k-u_{k-1},&q(u_k)-u_k
 		\end{bmatrix},
 	\end{equation*}
 	and 
 	\begin{equation*}
 		\mathcal{D}_k=
 		\begin{bmatrix}
 			r(u_{k-m_k+1})-r(u_{k-m_k}),& r(u_{k-m_k+2})-r(u_{k-m_k+1}),& \cdots & r(u_k)-r(u_{k-1}), &r(q(u_k))-r(u_k)
 		\end{bmatrix}.
 	\end{equation*}
 	Then, the least-squares problem  \eqref{eq:NGMRES-LSQ-simp} can be rewritten as
 	\begin{equation}\label{eq:Dk} 	
 		\min_{\bm{\gamma}^{(k)}}\|r(q(u_k))+\mathcal{D}_k\bm{\gamma}^{(k)}\|^2_2,
 	\end{equation}
 	which gives
 	\begin{equation}\label{eq:Dk-normeq} 
 		\bm{\gamma}^{(k)}=- (\mathcal{D}_k^T  \mathcal{D}_k)^{-1}\mathcal{D}_k^T r(q(u_k)).
 	\end{equation}
 	Define $e_k=u_k-u^*$. It follows that $r_k=Ae_k$. Moreover,
 	\begin{align*}
 		r(q(u_k))&=Aq(u_k)-b=A(Mu_k+b)-b=Mr(u_k)=MAe_k,\\
 		r(q(u_k))-r(u_k)&=Aq(u_k)-b-Au_k+b=A(q(u_k)-u_k)=-Ar(u_k)=-A^2e_k,\\
 		r(u_k)-r(u_{k-1})&=Au_k-b-Au_{k-1}+b=A(e_k-e_{k-1}).
 	\end{align*}
 	Based on the above formulas, we rewrite $\mathcal{S}_k$ and $\mathcal{D}_k$ as  
 	\begin{equation*}
 		\mathcal{S}_k=
 		\begin{bmatrix}
 			e_{k-m_k+1}-e_{k-m_k}& 	e_{k-m_k+2}-e_{k-m_k+1}, & \cdots,& e_k-e_{k-1},&	-Ae_k
 		\end{bmatrix},
 	\end{equation*}
 	and 
 	\begin{equation*}
 		\mathcal{D}_k=A
 		\begin{bmatrix}
 			e_{k-m_k+1}-e_{k-m_k}& 	e_{k-m_k+2}-e_{k-m_k+1}, & \cdots,& e_k-e_{k-1},&	-Ae_k
 		\end{bmatrix}=A\mathcal{S}_k.
 	\end{equation*}
 	Then, the update \eqref{NGMRES-update} of NGMRES($m$) is given by
 	\begin{align*}
 		u_{k+1}&=q(u_k)+\gamma_0^{(k)}(q(u_k)-u_k) + \sum_{i=1}^{m_k} \gamma^{(k)}_i\left(u_{k-i+1}- u_{k-i}\right)\\
 		&=u_k-r(u_k) +\mathcal{S}_k {\bm \gamma}^{(k)}\\
 		&=u_k-r(u_k)-\mathcal{S}_k (\mathcal{S}_k^T A^T A\mathcal{S}_k)^{-1}\mathcal{S}_k^TA^TM r(u_k)\\
 		&=u_k-(I + \mathcal{S}_k (\mathcal{S}_k^T A^T A\mathcal{S}_k)^{-1}\mathcal{S}_k^TA^TM) r(u_k),
 	\end{align*} 
 	which can be treated as a multisecant method. It follows that
 	\begin{equation*}
 		e_{k+1}=e_k-(I + \mathcal{S}_k (\mathcal{S}_k^T A^T A\mathcal{S}_k)^{-1}\mathcal{S}_k^TA^TM)Ae_k.
 	\end{equation*}
 	Let $W_k=\mathcal{S}_k (\mathcal{S}_k^T A^T A\mathcal{S}_k)^{-1}\mathcal{S}_k^TA^T$. Then,
 	\begin{align*}
 		&I-(I + \mathcal{S}_k (\mathcal{S}_k^T A^T A\mathcal{S}_k)^{-1}\mathcal{S}_k^TA^TM)A\\
 		=&I- (I+W_kM)(I-M) \\
 		=&(I-W_kA)M.
 	\end{align*}
 	Thus,
 	\begin{equation}\label{eq:ekplusone}
 		e_{k+1}=(I-W_kA)Me_k, \quad  r_{k+1}= (I-AW_k)Mr_k.
 	\end{equation}
 	From the above equalities, we see that the $(k+1)$th error or residual is related to the previous $m_k+1$ iterates, which makes the convergence analysis complicated. To better understand the performance of NGMRES, we will consider special $m_k+1$ iterates to quantitatively assess a single NGMRES($m$) in the next subsection.

 	\subsection{Improvement of NGMRES}
 	Here, we study how NGMRES($m$) can improve the performance of the fixed-point iteration given by \eqref{eq:FP}. In order to do this, we consider one-step NGMRES, that is, given $u_0$, we compute $u_j=q(u_{j-1})$, where $j=1,2,\cdots, k$, and then generate $u_{k+1}$ by NGMRES($m$). We compare $u_{k+1}$ with $q(u_k)$ which is the fixed-point iterate. 
 	
 	For later use, we first list the Chebyshev polynomial property as follows. We denote by $P_s$ the set of all polynomial with degree at most $s$. 
 	
 	\begin{lemma}\cite{saad2003iterative}\label{lem:chebyshe-property}
 		Let $[\alpha,\beta]$ be a non-empty interval in $\mathbb{R}$, which does not include $0$ and 1, and let $p_s(t)$ be a polynomial with degree at most $s$. Then,
 		\begin{equation*}
 			\min_{p_s\in P_s, p_s(\gamma)=1} \max_{t\in [\alpha,\beta]}|p_s(t)| =\frac{1}{\big|C_s(1+2\frac{\gamma-\beta}{\beta-\alpha})\big|},
 		\end{equation*}
 		where $C_s$  is the Chebyshev polynomial of degree $s$ of the first kind.
 	\end{lemma}
 	Let $a_1<a_2, \alpha=\frac{1}{a_2}$ and $\beta=\frac{1}{a_1}$ and $[\alpha,\beta]$ does not include $0$ and 1. Denote
 	\begin{equation*}
 		\epsilon(a_1,a_2,s)=	\min_{p_s\in P_s, p_s(1)=1} \max_{t\in [\alpha,\beta]}|p_s(t)| =\frac{1}{\big|C_s(1+2\frac{1-\beta}{\beta-\alpha})\big|}.
 	\end{equation*}
 We begin by examining the improvement under the assumption that $M$ is diagonalizable, followed by an analysis of the case where $M$ is symmetric and positive definite.
 	\subsubsection{Diagonalizable case}
 	We first consider the case that $M$ is diagonalizable, i.e., $M=U\Lambda U^{-1}$, where $\Lambda$ is a diagonal matrix. Assume that $u_j=Mu_{j-1}+b$ for $j=1,\cdots, k$. Then,  $e_j=M^je_0$ and $e_j-e_{j-1}=M^{j-1}(M-I)e_0=U\Lambda^{j-1}(\Lambda-I)U^{-1}e_0$.

 	Define $\mathcal{F}_{m_k+1}$ as the Krylov subspace matrix given by
 	\begin{equation*}
 		\mathcal{F}_{m_k+1}(\Lambda, U^{-1}e_0)=[ U^{-1}e_0, \Lambda U^{-1}e_0, \cdots, \Lambda^{m_k-1} U^{-1}e_0,  \Lambda^{m_k}U^{-1}e_0].
 	\end{equation*}
 	Then,
 	\begin{align*}
 		\mathcal{S}_k &=
 		\begin{bmatrix}
 			e_{k-m_k+1}-e_{k-m_k},& \cdots, &e_k-e_{k-1}, &-Ae_k 
 		\end{bmatrix}\\
 		&=[M^{k-m_k}(M-I)e_0,\cdots , M^{k-1}(M-I)e_0, -(I-M)M^ke_0]\\
 		&=-(I-M)M^{k-m_k}[e_0, Me_0, \cdots,  M^{m_k-1} e_0,  M^{m_k}e_0]\\
 		&=U(\Lambda-I)\Lambda^{k-m_k} [ U^{-1}e_0, \Lambda U^{-1}e_0, \cdots, \Lambda^{m_k-1} U^{-1}e_0,  \Lambda^{m_k}U^{-1}e_0]\\
 		&=U(\Lambda-I)\Lambda^{k-m_k} \mathcal{F}_{m_k+1}(\Lambda, U^{-1}e_0).
 	\end{align*}
 	It follows that
 	\begin{align*}
 		\mathcal{D}_k= A\mathcal{S}_k& = (I-M)U(\Lambda-I)\Lambda^{k-m_k} \mathcal{F}_{m_k+1}(\Lambda, U^{-1}e_0) \\
 		&=  -U(\Lambda-I)^2\Lambda^{k-m_k} \mathcal{F}_{m_k+1}(\Lambda, U^{-1}e_0)\\
 		&= -U\mathcal{F}_{m_k+1}(\Lambda, (\Lambda-I)^2\Lambda^{k-m_k} U^{-1}e_0)\\
 		&=-U\mathcal{H},
 	\end{align*}
 	where 
 	\begin{equation} \label{eq:def-H}
 		\mathcal{H} =\mathcal{F}_{m_k+1}(\Lambda, (\Lambda-I)^2\Lambda^{k-m_k} U^{-1}e_0)=\mathcal{F}_{m_k+1}(\Lambda, EU^{-1}e_0)
 	\end{equation}
 	with $E=(\Lambda-I)^2\Lambda^{k-m_k}$.
 	
 	By standard calculation, we have
 	\begin{align}
 		I-W_kA & = I-\mathcal{S}_k (\mathcal{S}_k^T A^T A\mathcal{S}_k)^{-1}\mathcal{S}_k^TA^T A\nonumber\\
 		& = I-A^{-1}A\mathcal{S}_k (\mathcal{S}_k^T A^T A\mathcal{S}_k)^{-1}\mathcal{S}_k^TA^T A\nonumber\\
 		&=A^{-1}(I- U\mathcal{H}(\mathcal{H}^TU^TU\mathcal{H})^{-1}\mathcal{H}^TU^T )A \nonumber\\
 		&=U(I-\Lambda)^{-1}U^{-1}(I- U\mathcal{H}(\mathcal{H}^TU^TU\mathcal{H})^{-1}\mathcal{H}^TU^T )U(I-\Lambda)U^{-1}\nonumber\\
 		&=U G_k U^{-1}, \label{eq:WkA}
 	\end{align}
 	where
 	\begin{equation}\label{eq:Gk-diag}
 		G_k= \Theta(I-\mathcal{H}(\mathcal{H}^TU^TU\mathcal{H})^{-1}\mathcal{H}^T U^T U)\Theta^{-1}, \quad \Theta=(I-\Lambda)^{-1}.
 	\end{equation}
 	
 	\begin{lemma}\label{lem:diag}
 		Assume that $M$ is diagonalizable, ie., $M=U\Lambda U^{-1}$, where $\Lambda$ is a diagonal matrix.  Then, the error $e_{k+1}$ obtained from one-step NGMRES($m$) satisfies
 		\begin{equation*}
 			e_{k+1}=U G_k U^{-1} Me_k,
 		\end{equation*}
 		where $G_k$ is defined in \eqref{eq:Gk-diag},  and the residual satisfies	
 		\begin{equation*}
 			r_{k+1}=R_k Mr_k,
 		\end{equation*}
 		where $R_k=  I-U\mathcal{H}(\mathcal{H}^TU^TU\mathcal{H})^{-1}\mathcal{H}^T U^T$.
 	\end{lemma}
 	
 	\begin{proof}
 		Using \eqref{eq:ekplusone}, \eqref{eq:WkA} and \eqref{eq:Gk-diag} gives the desired results.
 	\end{proof}
 	Based on the above analysis, we next present an error estimate for the residual of one-step NGMRES($m$) under the assumption that the eigenvalues of $M$ lie within a specified interval.
 	\begin{theorem}\label{thm:diag}
 		Assume that $M$ is a diagonalizable matrix, i.e., $M=U\Lambda U^{-1}$, where $\Lambda$ is a diagonal matrix, the eigenvalues of $M$ are contained in an interval $[a_1,a_2]$ that does not include $0$ and 1, and $m\leq k$. Then, the 2-norm of the residual $r_{k+1}$ obtained from one-step NGMRES($m$) satisfies
 		\begin{equation}\label{eq:error-est-M- diagonalizable}
 			\|r_{k+1} \|\leq \epsilon(a_1,a_2,m+1)\cdot \kappa_2(U)\cdot\| Mr_k\|.
 		\end{equation}
 		where $\kappa_2(U)=\|U\| \|U^{-1}\|$.
 	\end{theorem}
 	
 	Note that if the $(k+1)$th iterate is obtained by the fixed-point iteration, then the corresponding residual is $Mr_k$.  Our theoretical result \eqref{eq:error-est-M- diagonalizable} provides a quantitative measurement of how much the fixed-point iteration can be improved. In fact, the factor $\epsilon(a_1,a_2,m+1)\cdot \kappa_2(U)$ is the gain in how much one-step NGMRES($m$) improves the convergence speed of the fixed-point iteration. 
 	
 	\begin{proof}
 		Since $m\leq k$, $m_k=\min\{m, k\}=m$. From Lemma \ref{lem:diag} and the definition of $\mathcal{H}$ in \eqref{eq:def-H}, we have
 		\begin{align*}
 			\|r_{k+1} \| &=\|(I-U\mathcal{H}(\mathcal{H}^TU^T U\mathcal{H})^{-1}\mathcal{H}^TU^T)Mr_k \| \\
 			& =\min_{c \in \mathbb{R}^{m+1}}\|Mr_k -U\mathcal{H}c\|\\
 			& =\min_{p_m \in P_m}\|Mr_k -p_m(M)UEU^{-1}e_0 \|,
 		\end{align*}
 		where $p_m$ is a polynomial with degree at most $m$.  
 		
 		Note that $U^{-1}Me_k=U^{-1}M^{k+1}e_0=\Lambda^{k+1}U^{-1}e_0$. Thus,
 		\begin{align*}
 			UEU^{-1}e_0&= U(\Lambda-I)^2\Lambda^{k-m}U^{-1}e_0\\
 			&=U(I-\Lambda)\Lambda^{-m-1}(I-\Lambda)\Lambda^{k+1}U^{-1}e_0\\
 			&= U(I-\Lambda)\Lambda^{-m-1}(I-\Lambda) U^{-1}Me_k\\
 			&= U(I-\Lambda)\Lambda^{-m-1}U^{-1}U(I-\Lambda) U^{-1}Me_k\\
 			&= U(I-\Lambda)\Lambda^{-m-1}U^{-1}Mr_k.
 		\end{align*}
 		It follows that
 		\begin{align*}
 			\|r_{k+1}\| &=\min_{p_m \in P_m}\|\left(I- p_m(U\Lambda U^{-1}) U(I-\Lambda)\Lambda^{-m-1}U^{-1} \right)Mr_k\|\\
 			&\leq\min_{p_m \in P_m}\|I- p_m(U\Lambda U^{-1}) U(I-\Lambda)\Lambda^{-m-1}U^{-1}\| \|Mr_k\|\\
 			& \leq \|U\|\|U^{-1}\|\min_{p_m \in P_m}\|I- p_m(\Lambda)(I-\Lambda)\Lambda^{-m-1}\| \|Mr_k\|\\
 			& \leq \|U\|\|U^{-1}\|\min_{p_m\in P_m} \max_{t\in [a_1,a_2]}\left|1-p_m(t) t^{-m-1}(1-t)\right| \|Mr_k\|\\
 			& \leq \|U\|\|U^{-1}\|\min_{p_m \in P_m} \max_{t\in [1/a_2,1/a_1]}\left|1-p_m(1/t) t^{m}(t-1)\right| \|Mr_k\|\\
 			& \leq \|U\|\|U^{-1}\|\min_{p_m \in P_m} \max_{t\in [1/a_2,1/a_1]}\left|1-p_m(t)(t-1)\right| \|Mr_k\|\\
 			& \leq \|U\|\|U^{-1}\|\min_{q_{m+1}\in P_{m+1}, q_{m+1}(1)=1} \max_{t\in [1/a_2,1/a_1]}\left|q_{m+1}(t)\right| \|Mr_k\|.
 		\end{align*}
 		Using Lemma \ref{lem:chebyshe-property} for the right-hand side of the above inequality gives the desired result.
 	\end{proof}
 	Next, we present an error estimate for the residual of one-step NGMRES($m$), assuming the eigenvalues of $M$ lie within a generic domain in the complex plane.
 	\begin{theorem}\label{thm:diag-m=k}
 		Assume that $M$ is a diagonalizable matrix, i.e., $M=U\Lambda U^{-1}$, where $\Lambda$ is a diagonal matrix. Denote the eigenvalues of $M$ as $\Sigma(M)$. Assume that $m=k$.  Then, the 2-norm of the residual $r_{k+1}$ obtained from one-step NGMRES($m$) satisfies
 		\begin{equation}\label{eq:error-diag-m=k}
 			\|r_{m+1} \|\leq \chi^{(m+1)}\cdot \kappa_2(U)\cdot\| r_0\|.
 		\end{equation}
 		where $\kappa_2(U)=\|U\|\|U^{-1}\|$ and 
 		\begin{equation}\label{eq:def-chi}
 			\chi(m+1)=\min_{q_{m+1} \in P_{m+1}, q_{m+1}(1)=1} \max_{t\in\Sigma(M)}\left| q_{m+1}(t)\right|.
 		\end{equation}
 	\end{theorem}
 	
 	\begin{proof}
 		From the proof of Theorem \ref{thm:diag}, we have
 		\begin{align*}
 			Mr_k -p_m(M)UEU^{-1}e_0&= Mr_k -p_m(M)UEU^{-1}e_0\\
 			&=Mr_k- p_m(M)U(I-\Lambda)\Lambda^{-m-1}U^{-1}Mr_k\\
 			&=U\Lambda^{k+1}U^{-1}r_0-U(I-\Lambda)p_m(\Lambda)\Lambda^{k-m}U^{-1}r_0\\
 			&=U(\Lambda^{k+1}-(I-\Lambda)p_m(\Lambda)\Lambda^{k-m})U^{-1}r_0.
 		\end{align*}
 		Using $m=k$, we have
 		\begin{equation*}
 			Mr_m -p_m(M)UEU^{-1}e_0 =U(\Lambda^{m+1}-(I-\Lambda)p_m(\Lambda))U^{-1}r_0.
 		\end{equation*}
 		Thus,
 		\begin{align*}
 			\|r_{m+1} \| &=\min_{p_m \in P_m}\|Mr_m -p_m(M)UEU^{-1}e_0 \|\\
 			&\leq  \|U\|\|U^{-1}\| \min_{p_{m} \in P_{m}}\|\Lambda^{m+1}-(I-\Lambda)p_m(\Lambda)\| \|r_0\| \\
 			&\leq  \|U\|\|U^{-1}\|\min_{q_{m+1} \in P_{m+1}, q_{m+1}(1)=1} \max_{t\in\Sigma(M)}\left| q_{m+1}(t)\right| \|r_0\|,
 		\end{align*}
 		which is the desired result.
 	\end{proof}
 	
 	We note that the estimate \eqref{eq:error-diag-m=k} is the same as the residual norm of the  $(m+1)$th step of GMRES; see Proposition 6.32 in Saad's work~\cite{saad2003iterative} for further details. In Addition, Saad's work~\cite{saad2003iterative} presents the convergence behavior of GMRES when the eigenvalues of $A$ lie within an ellipse, which may offer useful insights for estimating \eqref{eq:def-chi}. However, we do not pursue a detailed estimation of \eqref{eq:def-chi}. In section \ref{sec:aNGMERS}, we will see that under certain conditions, the $(m+1)$th iteration, $u_{m+1}$, generated by aNGMRES($\infty,m+1$)---the method proposed in this work---is identical to the $(m+1)$th iteration of GMRES. In the next section, we discuss the estimation of errors and residuals for a special case of $M$.

 	\subsubsection{Symmetric and positive definite case}
 	In the following, we consider $M$ to be symmetric and positive definite (SPD), i.e., there exists a unitary matrix $U$ such that $M=U\Lambda U^T$, where $\Lambda$ is a diagonal matrix.  Then,  $e_{j+1}=M^{j+1}e_0$ and $e_{j+1}-e_j=M^j(M-I)e_0=U\Lambda^j(\Lambda-I)U^Te_0$.

 	\begin{lemma}\label{lem:residual-spd}
 		Assume that $M$ is SPD, i.e., $M=U\Lambda U^T$, where $U$ is a unitary matrix and $\Lambda$ is a diagonal matrix. Then, the error $e_{k+1}$ obtained from one-step NGMRES($m$) satisfies
 		\begin{equation*}
 			e_{k+1}=U \bar{G}_k U^T Me_k,
 		\end{equation*}
 		where $\bar{G}_k=\Theta \bar{R}_k \Theta^{-1}$ with $\bar{R}_k=  I-\mathcal{H}(\mathcal{H}^T\mathcal{H})^{-1}\mathcal{H}^T$, and the residual $r_{k+1}$ satisfies	
 		\begin{equation*}
 			r_{k+1}=U \bar{R}_k U^T Mr_k.
 		\end{equation*}
 	\end{lemma}
 	
 	\begin{proof}
 		Since $M=U\Lambda U^T$, where $U$ is a unitary matrix, replacing $U^{-1}$ by $U^T$ in Lemma \ref{lem:diag} gives the desired results.
 	\end{proof}

 	\begin{theorem}
 		Assume that $M$ is SPD, i.e., $M=U\Lambda U^T$, where $U$ is a unitary matrix and $\Lambda$ is a diagonal matrix,  the eigenvalues of $M$ are contained in an interval $[a_1,a_2]$ that does not include $0$ and 1, and $m<k$.  Then, the 2-norm of the error $e_{k+1}$ obtained from one-step NGMRES($m$) satisfies
 		\begin{equation}\label{eq:ek1-estimate-spd}
 			\|\Theta^{-1}U^Te_{k+1}\| \leq \epsilon(a_1,a_2,m+1)\cdot\|\Theta^{-1}U^TMe_k\|,
 		\end{equation}
 		and the 2-norm of  the residual $r_{k+1}$ satisfies
 		\begin{equation}\label{eq:rk-estimate-spd}
 			\|r_{k+1} \|\leq \epsilon(a_1,a_2,m+1)\cdot\|Mr_k\|.
 		\end{equation}
 		Moreover, if $m=k$, then the 2-norm of the residual $r_{k+1}$ obtained from one-step NGMRES($m$) satisfies
 		\begin{equation}\label{eq:error-sys-m=k}
 			\|r_{m+1} \|\leq \chi(m+1)\| r_0\|,
 		\end{equation}
 		where $\chi$ is defined in \eqref{eq:def-chi}.
 	\end{theorem}
 	
 	\begin{proof}
 		From Lemma \ref{lem:residual-spd}, we have
 		\begin{align*}
 			\|\Theta^{-1}U^Te_{k+1} \| &=\|(I-\mathcal{H}(\mathcal{H}^T\mathcal{H})^{-1}\mathcal{H}^T)\Theta^{-1}U^TMe_k \| \\
 			& =\min_{c \in \mathbb{R}^{m+1}}\| \Theta^{-1}U^TMe_k -\mathcal{H}c\|\\
 			& =\min_{p_m \in P_m}\| \Theta^{-1}U^TMe_k -p_m(\Lambda)EU^Te_0 \|,
 		\end{align*}
 		where $p_m$ is a polynomial with degree at most $m$.  
 		
 		Note that $U^TMe_k=U^TM^{k+1}e_0=\Lambda^{k+1}U^Te_0$. Thus,
 		\begin{align*}
 			EU^Te_0&= (\Lambda-I)^2\Lambda^{k-m}U^Te_0\\
 			&=(I-\Lambda)\Lambda^{-m-1}(I-\Lambda)\Lambda^{k+1}U^Te_0\\
 			&= (I-\Lambda)\Lambda^{-m-1}\Theta^{-1} U^TMe_k.
 		\end{align*}
 		It follows that
 		\begin{align*}
 			\|\Theta^{-1}U^Te_{k+1}\|& \leq \min_{p_m \in P_m}\|I-p_m(\Lambda)\Lambda^{-m-1}(I-\Lambda)\| \|\Theta^{-1}U^TMe_k\|\\
 			& \leq \min_{p_m\in P_m} \max_{t\in [a_1,a_2]}\left|1-p_m(t) t^{-m-1}(1-t)\right| \|\Theta^{-1}U^TMe_k\|\\
 			& \leq \min_{p_m \in P_m} \max_{t\in [1/a_2,1/a_1]}\left|1-p_m(1/t) t^{m}(t-1)\right| \|\Theta^{-1}U^TMe_k\|\\
 			& \leq \min_{p_m \in P_m} \max_{t\in [1/a_2,1/a_1]}\left|1-p_m(t)(t-1)\right| \|\Theta^{-1}U^TMe_k\|\\
 			& \leq \min_{q_{m+1}\in P_{m+1}, q_{m+1}(1)=1} \max_{t\in [1/a_2,1/a_1]}\left|q_{m+1}(t)\right| \|\Theta^{-1}U^TMe_k\|.
 		\end{align*}
 		Using Lemma \ref{lem:chebyshe-property} gives \eqref{eq:ek1-estimate-spd}.
 		
 		The second result can be directly obtained from Theorem \ref{thm:diag} using the fact that $\kappa_2(U)=1$. The last result, \eqref{eq:error-sys-m=k}, is obtained from Theorem \ref{thm:diag-m=k}.
 		
 	\end{proof}
 	
 	Note that if the $(k+1)$th iterate is obtained by the fixed-point iteration, then the corresponding residual is $Mr_k$.  Therefore, the factor $\epsilon(a_1,a_2,m+1)$ in \eqref{eq:rk-estimate-spd} provides a quantitative measurement of how much the fixed-point iteration can be improved.
 	
 	\begin{remark}
 		Note that one-step Anderson acceleration for the symmetric case has been studied  \cite{yang2022anderson}. The Chebyshev polynomial $\epsilon(a_1, a_2,m+1)$ for NGMRES($m$) is of degree $m+1$, but it is $m$ for AA($m$), and $\Theta$ is $\Lambda(\Lambda-I)^{-1}$.
 	\end{remark} 
 	
 Based on the one-step NGMRES($m$) described above, we next propose a new variant of NGMRES and conduct its convergence analysis.
 	
 	\section{Alternating NGMRES}\label{sec:aNGMERS}
 	
 	Note that NGMRES($m$) requires solving a least-squares problem at each iteration, which is computationally expensive for large $m$. To mitigate this, we apply NGMRES($m$) at periodic intervals within the fixed-point iteration.  We propose the following alternating NGMRES with depth $m$ and periodicity $p$, denoted by aNGMRES($m,p$) in Algorithm \ref{alg:aNGMRESm}. The main idea is that every $p-1$ steps of fixed-point iterations are followed by one step of NGMRES($m$). By doing this, we do not need to solve the least-squares problem at every iteration. In Algorithm \ref{alg:aNGMRESm}, $u_k$ defined in \eqref{eq:Alg-NG} is updated using \eqref{NGMRES-update}.  We notice that the same idea has been applied to alternating Anderson acceleration \cite{pratapa2016anderson,lupo2019convergence,suryanarayana2019alternating,pollock2019anderson}. There are two advantages of this new algorithm. Firstly, aNGMRES($m,p$) solves a least squares problem only at periodic intervals of length $p$, which reduces computational complexity and increases locality. Secondly, it may enhance the robustness against
 	stagnations, which is observed in NGMRES($m$) \cite{GreifHe25NGMRES}, by making their occurrence less likely.  We will explore connections between aNGMRES, GMRES and restarted GMRES (GMRES($m$)) where GMRES is restarted once the Krylov subspace reaches dimension $m$ and the current approximate solution becomes the new initial guess for the next $m$ iterations.
 	\begin{algorithm}[h]
 		\caption{Alternating NGMRES with depth $m$ and periodicity $p$: aNGMRES($m,p$)} \label{alg:aNGMRESm}
 		\begin{algorithmic}[1] 
 			\State Given $u_0, m\geq 0$
 			\For {$k= 1,2,\cdots$}
 			\If { mod($k,p$)=0,}
 			\State\begin{equation}\label{eq:Alg-NG}
 				u_k\leftarrow {\rm NGMRES}(m)
 			\end{equation}                 
 			\Else
 			\State\begin{equation*}
 				u_k\leftarrow q(u_{k-1})
 			\end{equation*}          
 			\EndIf
 			\EndFor
 		\end{algorithmic}
 	\end{algorithm}
 	Note that when $p=1$, aNGMRES($m,p$) is NGMRES($m$).  When $m=\infty$, then \eqref{eq:Alg-NG} is NGMRES($k-1$), which uses all previous iterates to compute $u_k$. We comment that in Algorithm \ref{alg:aNGMRESm}, the fixed-point iteration operator $q$ can be any iterative method. However, in this work, we limit ourselves to the Richardson iteration \eqref{eq:FP} for the purpose of convergence analysis.
 	
 	\subsection{Comparison between aNGMRES, GMRES, and restarted GMRES}
 	To study the relationship between aNGMRES($m,p$), GMRES, and restarted GMRES, we first introduce some terminologies. We denote the Krylov subspace associated with $A$ and the initial residual $r_0$ by
 	\begin{equation*}
 		\mathcal{K}_k(A, r_0)={\rm span}\{r_0, Ar_0,\cdots, A^{k-1}r_0\}.
 	\end{equation*}
 	
 	Following the work of Potra and Engler \cite{potra2013characterization}, we define the stagnation index of full GMRES as follows.
 	\begin{equation*}
 		\eta^G=\min\left\{\ell \in \mathbb{N}: u_{\ell}^G=u_{\ell+1}^G\right\}.
 	\end{equation*}
 	
 	The grade of $r_0\neq 0$ with respect to $A$ is defined as \cite{potra2013characterization}
 	\begin{align*}
 		\nu=\nu(A,r_0)&=\max\{\ell\in\mathbb{N}: \dim\mathcal{K}_{\ell}=\ell\}\\
 		& =\min\left\{\ell\in\mathbb{N}: r_0, Ar_0, \cdots, A^{\ell}r_0 \quad\text{are linearly dependent}\right\}.
 	\end{align*}

 	\begin{lemma}\label{lem:FP-kry}
 		Let $u_k=q(u_{k-1})=Mu_{k-1}+b$ for $k=1, 2,\cdots$. Then,
 		\begin{equation*}
 			\mathcal{K}_{j}(A,r_0)={\rm span}\{u_1-u_0, u_2-u_0,\cdots, u_j-u_0\}, \forall j\in\mathbb{N}.
 		\end{equation*}
 	\end{lemma}
 	\begin{proof}
 		We proceed with the proof by induction.  Note that
 		\begin{equation}\label{eq:u1u0}
 			u_1=Mu_0+b=u_0-r_0.
 		\end{equation}
 		It is obvious that $\text{span} \{u_1-u_0 \}= \mathcal{K}_{1}(A,r_0)$. Assume that for $j\leq s$, we have
 		\begin{equation*}
 			\mathcal{K}_{j}(A,r_0)=\text{span}\{u_1-u_0, u_2-u_0,\cdots, u_j-u_0\}.
 		\end{equation*} 
 		Let $w_j=u_j-u_0$. It follows that $w_j \in \mathcal{K}_{j}(A,r_0)$. Next, we consider $j=s+1$. Note that $$u_{s+1}=q(u_s)=Mu_s+b=(I-A)u_s+b=u_s+b-A(u_0+w_s)=u_s-r_0-Aw_s.$$ Thus,
 		\begin{equation*}
 			u_{s+1}-u_0=u_s-u_0-r_0-Aw_s=w_s-r_0-Aw_s\in\mathcal{K}_{s+1}(A,r_0).
 		\end{equation*}
 		It follows that
 		\begin{equation*}
 			\Delta:=\text{span}\{u_1-u_0, u_2-u_0,\cdots, u_{s+1}-u_0\} \subset \mathcal{K}_{s+1}(A,r_0).
 		\end{equation*} 
 		Next, we prove $\mathcal{K}_{s+1}(A,r_0)\subset \Delta $, i.e., that $A^sr_0 \in \Delta$.
 		
 		Since $\Delta \subset \mathcal{K}_{s+1}(A,r_0)$, for $k=1, 2, \cdots,s+1$,
 		\begin{equation}\label{eq:Delta-Ks1}
 			u_k =u_0+\sum_{i=0}^{k-1}c_{k,i}A^{i}r_0.
 		\end{equation}
 		We claim that 
 		\begin{equation}\label{eqckk}
 			c_{k,k-1}=(-1)^k, \quad k=1,2, \dots, s+1.
 		\end{equation}
 		For $k=1$, from  \eqref{eq:u1u0}, equation \eqref{eqckk} holds. Assume that for $k\leq s$, $c_{k,k-1}=(-1)^k$.	
 		Then
 		\begin{align*}
 			u_{s+1}=q(u_s)&=u_s-r_s\\
 			&=u_0+\sum_{i=0}^{s-1}c_{s,i}A^{i}r_0 +b-A(u_0+\sum_{i=0}^{s-1}c_{s,i}A^{i}r_0)\\
 			&=u_0-r_0+(I-A)\sum_{i=0}^{s-1}c_{s,i}A^{i}r_0.
 		\end{align*}
 		Thus, $c_{s+1,s}$, the coefficient of $A^sr_0$ in $u_{s+1}$ is $-c_{s,s-1}=-(-1)^{s}=(-1)^{s+1}$.
 		
 		Next, we show that there exist $\{x_i\}_{i=1}^{s+1}$ such that $A^sr_0=\sum_{i=1}^{s+1}x_i(u_i-u_0)$. Using \eqref{eq:Delta-Ks1} gives
 		\begin{equation*}
 			\begin{bmatrix}
 				c_{1,0} & c_{2,0} & c_{3,0} & \cdots & c_{s,0} & c_{s+1,0} \\
 				0      & c_{2,1} & c_{3,1} & \cdots & c_{s,1} & c_{s+1,1} \\
 				0      & 0      & c_{3,2} & \cdots & c_{s,2} & c_{s+1,2} \\
 				\vdots & \vdots & \vdots & \ddots   &\vdots & \vdots \\
 				0      & 0      & 0      & \cdots & c_{s,s-1} & c_{s+1,s-1}\\
 				0      & 0      & 0      & \cdots &0  & c_{s+1,s}
 			\end{bmatrix}	
 			\begin{bmatrix}
 				x_1 \\ x_2 \\x_3 \\ \vdots \\x_s \\ x_{s+1}
 			\end{bmatrix}=
 			\begin{bmatrix}
 				0 \\ 0 \\0 \\ \vdots \\ 0\\ 1
 			\end{bmatrix}.
 		\end{equation*}
 		From \eqref{eqckk}, we know that the above system has a solution. It follows that $A^sr_0 \in \Delta$. We have $\Delta = \mathcal{K}_{s+1}(A,r_0)$, which is the desired result.
 	\end{proof}
 	Lemma \ref{lem:FP-kry} establishes a connection between Krylov subspace of GMRES and the space associated with fixed-point iterations, which is linked to the periodic iterates of aNGMRES. If we consider aNGMRES($m,p$) with finite $m$ and $p=m+1$ applied to $q(u)=Mu+b$, then aNGMRES($m,m+1$) and the restarted GMRES (GMRES($m+1$))  are identical at the end of each periodic interval of length $m+1$. We state it in the following theorem.
 	
 	\begin{theorem}\label{thm:aNGMRES=rGMRES}
 		Consider aNGMRES($m,m+1$) presented in Algorithm \ref{alg:aNGMRESm} applied to $q(u)=Mu+b$. Assume that aNGMRES($m,m+1$) and GMRES($m+1$) use the same initial guess $u_0$ and let the corresponding sequences generated by them be denoted as $\left\{u_k \right\}$ and $\left\{u_k^{G{(m+1)}}\right\}$. Then, 
 		\begin{equation*}
 			u_{j(m+1)} =u_{j(m+1)}^{G(m+1)}, \quad  j=1, 2, \cdots
 		\end{equation*}
 	\end{theorem}
 The key idea of the proof is to examine the Krylov subspace associated with the aNGMRES iterates and to use the property that the $k$th NGMRES iterate minimizes the residual norm.
 	\begin{proof}
 		For aNGMRES($m,m+1$), the iterates
 		\begin{equation*}
 			\left\{u_{j(m+1)+1}, u_{j(m+1)+2}, \cdots, u_{(j+1)(m+1)}\right\},\quad j=0, 1, 2,\cdots
 		\end{equation*}
 		are constructed by first applying $m$ fixed-point iterations to $u_{j(m+1)}$, obtaining the first $m$ iterates. Then, $u_{(j+1)(m+1)}$ is obtained by applying NGMRES($m$) to $\{u_{j(m+1)}, u_{j(m+1)+1},\cdots, u_{j(m+1)+m}\}$.  The next $(m+1)$ iterates are obtained by repeating this procedure starting with $u_{(j+1)(m+1)}$. Thus, we only need to show  $u_{m+1} =u_{m+1}^{G(m+1)}$. 
 		
 		For aNGMRES($m,m+1$), we know that $u_i=q(u_{i-1})$ for $i=1, 2,\cdots, m=p-1$. Define $w_i=u_i-u_0, i=1, 2, \cdots, m$, and $w_p=q(u_{p-1})-u_0$. Note that $p=m+1$. At the $p$th iteration of aNGMRES($m,m+1$), we have
 		\begin{align*}
 			u_{p}&=q(u_{p-1})+\sum_{i=0}^{p-1} \beta^{(p-1)}_i\left(q(u_{p-1})- u_{p-1-i})\right)\\
 			&=(1+\sum_{i=0}^{p-1} \beta^{(p-1)}_i)(q(u_{p-1})-u_0)+u_0-\sum_{i=0}^{p-1} \beta^{(p-1)}_i\left(u_{p-1-i}-u_0\right)\\
 			&=(1+\sum_{i=0}^{p-1} \beta^{(p-1)}_i)w_p+u_0-\sum_{i=0}^{p-2} \beta^{(p-1)}_iw_{p-1-i}\\
 			&=u_0+\sum_{i=1}^{p} \tau^{(p-1)}_iw_i,
 		\end{align*}
 		where $\tau_i=-\beta^{(p-1)}_{p-1-i}$ for $i=1, 2, \cdots, p-1$ and $\tau_p=(1+\sum_{i=0}^{p-1} \beta^{(p-1)}_i)$. Let ${\bf \tau}=[\tau_1, \tau_2,\cdots, \tau_p]^T$. From Lemma \ref{lem:NGMRES-mini-resi}, ${\bf \tau}$ solves the following minimization problem
 		\begin{equation*}
 			\min_{{\bf\tau}\in\mathbb{R}^{p}}	\|Au_p-b\|=\min_{{\bf\tau}\in\mathbb{R}^{p}}\|A(u_0+\sum_{i=1}^{p} \tau^{(p-1)}_iw_i)-b\|=\min_{w \in \mathcal{K}_p(A,r_0)}\|A(u_0+w)-b\|,
 		\end{equation*}
 		where the last equality uses Lemma \ref{lem:FP-kry}.
 		According to the definition of GMRES algorithm, we have $u_p=u_p^{G}=u_p^{G(p)}$.
 	\end{proof}
 	
 	The above theorem suggests that aNGMRES($m,m+1$) can be treated as an alternative to restarted GMRES.  Note that when $m=0, p=1$,  aNGMRES($m,p$) is NGMRES(0), which is the same as GMRES(1) \cite{GreifHe25NGMRES}, i.e., $u_k=u_k^{G(1)}$ for all $k$, as stated in Theorem \ref{thm:aNGMRES=rGMRES}.

 	Next, we consider $m=\infty$.  We will show that aNGMRES($\infty,p$) and full GMRES  are identical at the end of each periodic interval of length $p$.  
 	\begin{theorem}\label{aNGMRES-GMRES}
 		Consider  aNGMRES($\infty,p$) presented in Algorithm \ref{alg:aNGMRESm} applied to $q(u)=Mu+b$. Assume that aNGMRES($\infty,p$) and full GMRES use the same initial guess $u_0$ and denote the corresponding sequences generated by them as $\left\{u_k \right\}$ and $\left\{u_k^{G }\right\}$, respectively. Then, 
 		\begin{equation}\label{eq:NG=G-jp}
 			u_{jp} =u_{jp}^{G}, \quad  j=1, 2, \cdots, j_*,
 		\end{equation}	
 		where $j_*p\leq \nu(A,r_0)$ and $j_*p\leq \eta^G$.
 		
 		Moreover, if the stagnation of GMRES does not occur, then there exists a $s\in\mathbb{N}$ such that
 		\begin{equation*}
 			u_{sp} =u^G_{\nu(A,r_0)}=u^*,  \quad \text{where}\quad (s-1)p<\nu(A,r_0)\leq sp.
 		\end{equation*}
 	\end{theorem}
 	\begin{proof}
 		Define
 		\begin{equation*}
 			\mathcal{K}_{jp}^{\alpha}=\text{span}\left\{u_1-u_0, u_2-u_0, \cdots, u_{jp-1}-u_0, u_{jp}-u_0\right\},  \, j=1, 2, \cdots
 		\end{equation*}
 		Note that the first $p$ iterates, i.e., $\{u_j\}_{j=1}^p$, of aNGMRES($\infty,p$) are the same as these of aNGMRES($p-1,p$). 
 		From Theorem \ref{thm:aNGMRES=rGMRES}, we have
 		\begin{equation*}
 			u_p=u_p^G.
 		\end{equation*}
 		Since $p<\nu(A,r_0)$ and $p<\eta^G$, we have $\mathcal{K}_p^{\alpha}=\mathcal{K}_p(A,r_0)$.
 		We proceed with the proof by induction.
 		Assume that for $j\leq j_*$ the $(j-1)p$th iteration satisfies
 		\begin{equation}\label{eq:jm1equl}
 			u_{(j-1)p} =u_{(j-1)p}^{G},
 		\end{equation}
 		and 
 		\begin{equation*}
 			\mathcal{K}_{(j-1)p}^{\alpha}=\mathcal{K}_{(j-1)p}(A,r_0).
 		\end{equation*}
 		Next, we demonstrate that the above two equations hold for the index $jp$. From \eqref{eq:jm1equl} and $(j-1)p<\nu$, we have
 		\begin{equation*}
 			u_{(j-1)p} =u_0+w_{(j-1)p}, 
 		\end{equation*}
 		where
 		\begin{equation}\label{eq:multiplyA}
 			w_{(j-1)p}\in \mathcal{K}_{(j-1)p}(A,r_0) \quad \text{and}\quad w_{(j-1)p}\notin \mathcal{K}_{(j-1)p-1}(A,r_0).
 		\end{equation}
 		After the $(j-1)p$th iteration, we perform $p-1$ successive Richardson iterations.
 		For $i=1,2,\cdots, p-1$, define $u_{(j-1)p+i}=u_0+w_{(j-1)p+i}$. Then, for $i=1$
 		\begin{align*}
 			u_{(j-1)p+1}&=q(u_{(j-1)p})\\
 			&=(I-A)u_{(j-1)p}+b\\
 			&=(I-A)(u_0+w_{(j-1)p})+b\\
 			&=u_0-r_0+w_{(j-1)p}-Aw_{(j-1)p}.
 		\end{align*}
 		Therefore,
 		\begin{equation*}
 			w_{(j-1)p+1}=-r_0+w_{(j-1)p}-Aw_{(j-1)p}\in \mathcal{K}_1(A,r_0)+\mathcal{K}_{(j-1)p}(A,r_0) +A\mathcal{K}_{(j-1)p}(A,r_0).
 		\end{equation*}
 		From \eqref{eq:multiplyA}, we have
 		\begin{equation*}
 			w_{(j-1)p+1}\in \mathcal{K}_{(j-1)p+1}(A,r_0) \quad \text{and}\quad w_{(j-1)p+1}\notin \mathcal{K}_{(j-1)p}(A,r_0).
 		\end{equation*}
 		For $i=2,\cdots, p-1$,  repeating the same procedure leads to 
 		\begin{equation*}
 			w_{(j-1)p+i}\in \mathcal{K}_{(j-1)p+i}(A,r_0) \quad \text{and}\quad w_{(j-1)p+i}\notin \mathcal{K}_{(j-1)p+i-1}(A,r_0),
 		\end{equation*}
 		and 
 		\begin{equation*}
 			\left(q(u_{jp-1})-u_0\right)\in \mathcal{K}_{jp}(A,r_0) \quad \text{and}\quad \left(q(u_{jp-1})-u_0\right)\notin \mathcal{K}_{jp-1}(A,r_0).
 		\end{equation*}
 		Let
 		\begin{equation*}
 			\mathcal{\bar{K}}_{jp}=\text{span}\left\{u_1-u_0, u_2-u_0, \cdots, u_{jp-1}-u_0, q(u_{jp-1})-u_0\right\}.
 		\end{equation*}
 		Since $jp\leq\nu(A,r_0)$ and $jp\leq\eta^G$, the dimension of $\mathcal{K}_{jp}(A,r_0)$ is $jp$. Then, $\mathcal{\bar{K}}=\mathcal{K}_{jp}(A,r_0)$.
 		Next, we compute the  $(jp)$th iteration of  aNGMRES($\infty,p$), which is given by NGMRES($jp$) 
 		\begin{align*}
 			u_{jp}&=q(u_{jp-1})+\sum_{i=0}^{jp-1} \beta^{(jp-1)}_i\left(q(u_{jp-1})- u_{jp-1-i})\right)\\
 			&=(1+\sum_{i=0}^{jp-1} \beta^{(jp-1)}_i)(q(u_{jp-1})-u_0)+u_0-\sum_{i=0}^{jp-1} \beta^{(jp-1)}_i\left(u_{jp-1-i}-u_0\right)\\
 			&=(1+\sum_{i=0}^{jp-1} \beta^{(jp-1)}_i)(q(u_{jp-1})-u_0)+u_0-\sum_{i=0}^{jp-2} \beta^{(jp-1)}_i\left(u_{jp-1-i}-u_0\right)\\
 			&=u_0+\sum_{i=1}^{jp} \tau^{(jp-1)}_iz_i,
 		\end{align*}
 		where $\tau_i=-\beta^{(jp-1)}_{jp-1-i}, z_i=u_i-u_0$ for $i=1, 2, \cdots, jp-1$, and  $\tau_{jp}=(1+\sum_{i=0}^{jp-1} \beta^{(jp-1)}_i), z_{jp}=q(u_{jp-1})-u_0$. Let ${\bf \tau}=[\tau_1, \tau_2,\cdots, \tau_{jp}]^T$. Then $\tau$ solves
 		\begin{equation}\label{eq:jpNGMRES}
 			\min_{{\bf\tau}\in\mathbb{R}^{jp}}\|Au_{jp}-b\|=\min_{{\bf\tau}\in\mathbb{R}^{jp}}	\|A(u_0+\sum_{i=1}^{jp} \tau^{(jp-1)}_iz_i)-b\|=\min_{z \in \mathcal{\bar{K}}}\|A(u_0+z)-b\|.
 		\end{equation}
 		Thus, $u_{jp}=u_{jp}^G$.
 		
 		Next, we consider the situation that stagnation does not occur. The index $\nu(A,r_0)$ would be either $sp$ for some integer $s$ or $(s-1)p<\nu(A,r_0)<sp$. 
 		
 		If $\nu(A,r_0)=sp$, then GMRES converges to the exact solution at $(sp)$th iteration, i.e., $u_{sp}=u_{sp}^G=u^*$.
 		
 		If $(s-1)p<\nu(A,r_0)<sp$, then GMRES converges to the exact solution at $\nu$th iteration, and 
 		\begin{equation*}
 			\mathcal{\bar{K}_{\nu}}=\text{span}\left\{u_1-u_0, u_2-u_0, \cdots, u_{(s-1)p}-u_0, u_{(s-1)p+1}-u_0, \cdots, u_{\nu}-u_0\right\}=\mathcal{K_{\nu}}(A,r_0).
 		\end{equation*}
 		After the $(s-1)p$th iteration,  aNGMRES($\infty,p$) processes $(p-1)$ Richardson iterations, and  $\mathcal{\bar{K}}_{sp}$ has the same dimension as $\mathcal{K}_{\nu}(A,r_0)$. Similarly to \eqref{eq:jpNGMRES}, the $(sp)$th iteration of aNGMRES($\infty,p$) is given by
 		\begin{equation*}
 			u_{sp}=q(u_{sp-1})+\sum_{i=0}^{sp-1} \beta^{(sp-1)}_i\left(q(u_{sp-1})- u_{sp-1-i})\right).
 		\end{equation*}
 		Let $\bm \beta^{(sp-1)}=[\beta^{(sp-1)}_0, \beta^{(sp-1)}_1,\cdots, \beta^{(sp-1)}_{sp-1}]^T$, which solves	 
 		\begin{equation*}
 			\min_{\bm \beta^{(sp-1)}}	\|Au_{sp}-b\|=\min_{z\in\mathcal{\bar{K}}_{sp}}\|A(u_0+z)-b\|.
 		\end{equation*}
 		It follows that $u_{sp}=u^G_{\nu}=u^*$. We have completed the proof.
 	\end{proof}
 	
 	We note that when $p=1$, aNGMRES($\infty$,1) reduces to NGMRES($\infty$). As shown in our recent work \cite{GreifHe25NGMRES},   NGMRES($\infty$) is identical to GMRES if the norm of the residuals of GMRES is strictly decreasing. For $p=1$, our result in \eqref{eq:NG=G-jp} is consistent with recent findings \cite{GreifHe25NGMRES}.  Recently, Lupo Pasini \cite{lupo2019convergence} proposed Anderson-type acceleration of Richardson's iteration (AAR) and showed that the $(jp)$th iteration by full AAR ($m=\infty$) after the Anderson mixing is the same as the $(jp)$th GMRES iteration. Now, our theorem makes connection with AAR due to the equivalence to GMRES.
 	
 	When GMRES stagnation occurs, the convergence behavior of aNGMRES becomes complex. We do not explore this direction further from a theoretical perspective. Instead, we examine it numerically; see the numerical results of Example \ref{ex:blockCirex} in subsection \ref{sub:validation}.
 	
 	\subsection{Convergence analysis of aNGMRES}
 	In the following, we discuss the convergence of aNGMRES. We first consider the fixed-point iteration to be contractive. 
 	\begin{theorem}
 		Let $q(u)=Mu+b$. Consider  aNGMRES($m,p$) presented in Algorithm \ref{alg:aNGMRESm}. Assume that $\|M\|=c<1$. Then, the residual norm of aNGMRES($m,p$) ($m=\infty$ or any integer) converges to zero for any initial guess.
 	\end{theorem}
 	\begin{proof}
 		If $u_k$ is obtained from the fixed-point iteration, then $r_k=Mr_{k-1}$, and $\|r_k\|\leq c\|r_{k-1}\|$. If $u_k$ is obtained from a one-step NGMRES($m$), then from \eqref{eq:NG-minrk}, we have $\|r_k\|\leq \|r(q(u_{k-1}))\|=\|Mr_{k-1}\|\leq c \|r_{k-1}\|$. Thus, we have $\|r_k\|\leq c\|r_{k-1}\|$ for every $k$. Because $c<1$, the sequence $\{\|r_k\|\}$ converges to zero.
 	\end{proof}
 	The above result emphasizes on the convergence, but does not show by how much that aNGMRES can improve the convergence speed of the underlying fixed-point iteration. In the following, we derive convergence bounds for some cases.
 	In fact, we observe that even when $M$ is not contractive, aNGMRES may converge.

 	\begin{theorem}
 		Let $q(u)=Mu+b$. Consider  aNGMRES($m,p$) with $m< p-1$ presented in Algorithm \ref{alg:aNGMRESm}.
 		When $M$ is diagonalizable, we have the following error estimate 
 		\begin{equation*}
 			\|r_{jp}\|\leq \left(\epsilon(a,b, m+1) \|M\|^p\kappa_2(U)\right)^j \cdot \|r_0\|.
 		\end{equation*}
 		When $M$ is SPD, we have the following error estimate
 		\begin{equation}\label{eq:aNGMRES-mlesp}
 			\|r_{jp} \| \leq \left(\epsilon(a,b, m+1) \|M\|^p\right)^j\cdot \|r_0\|.
 		\end{equation}
 		
 		If $p=m+1$ and $M$ is diagonalizable, we have the following error estimate 
 		\begin{equation*}
 			\|r_{j(m+1)}\|\leq \left(\chi(m+1) \kappa_2(U)\right)^j \cdot \|r_0\|.
 		\end{equation*}
 		If $p=m+1$ and $A$ is SPD, we have the following error estimate
 		\begin{equation}\label{eq:aNGMRES-SPD}
 			\|r_{j(m+1)} \|\leq \chi(m+1)^j\|r_0\|  \leq \left(2 \left(\frac{\sqrt{\kappa}-1}{\sqrt{\kappa}+1}\right)^{m+1}\right)^j \|r_0\|,
 		\end{equation}
 		where $\kappa=\frac{\lambda_{\rm max}(A)}{\lambda_{\rm min}(A)}$.
 	\end{theorem}
 	\begin{proof}
 		These results can be derived directly from \eqref{eq:error-est-M- diagonalizable}, \eqref{eq:error-diag-m=k} and \eqref{eq:error-sys-m=k}. The second inequality in \eqref{eq:aNGMRES-SPD} can be derived easily by following Section 6.11.3 in the work of Saad \cite{saad2003iterative}.   Therefore, we omit it.
 	\end{proof}
 	
 	When $M$ is diagonalizable and $m<p-1$, if $\epsilon(a,b, m+1) \|M\|^p\kappa_2(U)<1$, then the subsequence $\left\{r_{jp}\right\}$ of $\left\{r_k\right\}$ converges to zero. When $M$ is SPD, it is possible that the subsequence of aNGMRES can converge provided that $\epsilon(a,b, m+1) \|M\|^p<1$. If $\|M\|<1$, then either $\big(\epsilon(a,b, m+1)\kappa_2(U)\big)^j$ or $\big(\epsilon(a,b, m+1)\big)^j$ is the gain obtained by aNGMRES at $jp$ iteration compared to the fixed-point iteration.
 	
 	When $M$ is diagonalizable and $p=m+1$, if $\chi(m+1)\kappa_2(U)<1$, then the subsequence $\left\{r_{j(m+1)}\right\}$ of $\left\{r_k\right\}$ converges to zero. When $A$ is SPD, we can choose $m$ sufficiently large such that 
 	$2 \left(\frac{\sqrt{\kappa}-1}{\sqrt{\kappa}+1}\right)^{m+1}<1$, then the subsequence $\left\{r_{j(m+1)}\right\}$ of $\left\{r_k\right\}$ converges to zero. The above theorem suggests that both the condition number of $A$ and the window size $m$ used in aNGMRES influence the rate of convergence of aNGMRES. In particular, a smaller condition number $\kappa$ is associated with improved algorithmic performance.
 	
 		\begin{remark}
 			As previously noted, aNGMRES can be applied to general fixed-point iterations. In this work, however, we restrict our convergence analysis to the specific form $q(x)=x+(b-Ax)$. It would be interesting to extend this analysis to the more general case 
 			$q(x)=x+P(b-Ax)$,  where $P$ serves as a preconditioner---leading to a preconditioned version of aNGMRES. A natural question that arises is the relationship between preconditioned aNGMRES and preconditioned GMRES algorithm, including its restarted variants. However, this investigation is beyond the scope of the present work.
 		\end{remark}
 	
 	
 	\section{Numerical results}\label{sec:num}
 	
 In this section, we begin by discussing practical aspects of using aNGMRES, including the selection of the window size $m$, possible strategies for solving the associated least-squares problems, and the choices of norms used in those problems. We then present some examples to confirm our theoretical findings and further explore the convergence behavior of aNGMRES in cases where GMRES stagnates. In our numerical test, we use MATLAB's backslash operator as the solver for the least-squares problems. Finally, we demonstrate the effectiveness of aNGMRES in solving a challenging nonlinear PDE boundary-value problem. A comparison with Anderson acceleration is provided.
 	
 		\subsection{Practical considerations}
 		In practice, the choice of $m$ in Algorithm \ref{alg:NGMRESm} is often problem-dependent; however, it should remain modest in size---generally within a few tens, say between 1 and 20---to ensure efficiency and stability. A large value of $m$ increases the memory requirements for solving the associated least-squares problem (LSP) and may lead to ill-conditioned systems. We note that alternative norms, such as the $\ell^1$ or $\ell^{\infty}$ norms, can be used in the optimization problem \eqref{eq:NGMRES-LSQ}. In this work, we focus on the $\ell^{2}$ norm, under which the minimization problem can be reformulated as a linear least-squares problem and solved efficiently. A thorough examination of the parameter $m$ and the impact of different norm choices is reserved for future research.
 		
 		Choosing an appropriate solution method for the LSP, \eqref{eq:NGMRES-LSQ}, is a key consideration in implementing   NGMRES algorithm. Although NGMRES was originally proposed in 1997, to the best of our knowledge, its practical implementation has not been thoroughly discussed in the literature. As the iterates $u_k$ approach the exact solution $u^*$, the matrix $\mathcal{D}_k$ in \eqref{eq:Dk} may become ill-conditioned. In this case, solving the normal equations corresponding to \eqref{eq:Dk} or using \eqref{eq:Dk-normeq} is undesirable, as this approach squares the condition number and can lead to numerical instability. 
 		
 		Notably, Anderson acceleration (AA) \cite{anderson1965iterative}  also involves solving an LSP and the issue of ill-conditioning has been extensively addressed in that context. For example, Fang and Saad \cite{fang2009two} suggest using rank-revealing QR decomposition or a truncated singular value decomposition. Another strategy involves QR factor-updating techniques combined with adaptive control of the window size parameter $m$ to drop columns if the coefficient matrix of the LSP becomes too ill-conditioned. Moreover, recent work has proposed new approaches to improve the efficiency of AA. For instance, Lockhart et al. \cite{lockhart2022performance} introduced three low-synchronization orthogonalization algorithms within AA, which reduce the number of global reductions per iteration to a constant of 2 or 3, independent of the parameter $m$.  Parallel implementations of AA have been investigated \cite{loffeld2016considerations}. These techniques appear well-suited for adaptation within the NGMRES framework and may facilitate more robust and efficient implementations, which we intend to explore in future research.

 	\subsection{Validation}\label{sub:validation}
 	We validate Theorems \ref{thm:aNGMRES=rGMRES} and \ref{aNGMRES-GMRES} through the following examples.
 	\begin{example}\label{ex:cirex}
 		Consider $A\in \mathbb{R}^{n\times n}$, where the -1th diagonal elements are one, the last element of the first row is one, and other elements are zero, and $Au=b$, where
 		\begin{equation}\label{eq:cirmatrix}
 			A = 
 			\begin{bmatrix}
 				0 & 0 & 0 & \cdots & 0 & 1 \\
 				1 & 0 & 0 & \cdots & 0 & 0 \\
 				0 & 1 & 0 & \cdots & 0 & 0 \\
 				\vdots & \vdots & \vdots & \ddots & \vdots & \vdots \\
 				0 & 0 & 0 & \cdots & 1 & 0
 			\end{bmatrix}, \quad 
 			b=\begin{bmatrix}
 				1\\0\\0\\\vdots\\0
 			\end{bmatrix}.
 		\end{equation}
 	\end{example}
 	
 	The matrix $A$ is invertible and the exact solution is $u^*=[0, 0, 0, \cdots, 0, 1]^T$.
 	Consider the initial guess $u_0$ to be an $n$-dimensional column vector consisting entirely of ones. It can be shown that $\nu(A,r_0)=n$.  We consider $n=36$ in this example. Note that in this example $\|M\|=\|I-A\|>1$ and is not invertible. This means that the fixed-point iteration does not converge. 
 	
 	To confirm Theorem \ref{thm:aNGMRES=rGMRES}, we take $m=3$ and $p=m+1=4$. We report the convergence history of aNGMRES(3,4) and GMRES(4). From Theorem \ref{thm:aNGMRES=rGMRES}, we know that the $(4j)$th iterate of the two methods is the same, which is confirmed in Figure \ref{fig:aNG=rGMRES}.
 	\begin{figure}[h]
 		\centering
 		\includegraphics[width=8.5cm]{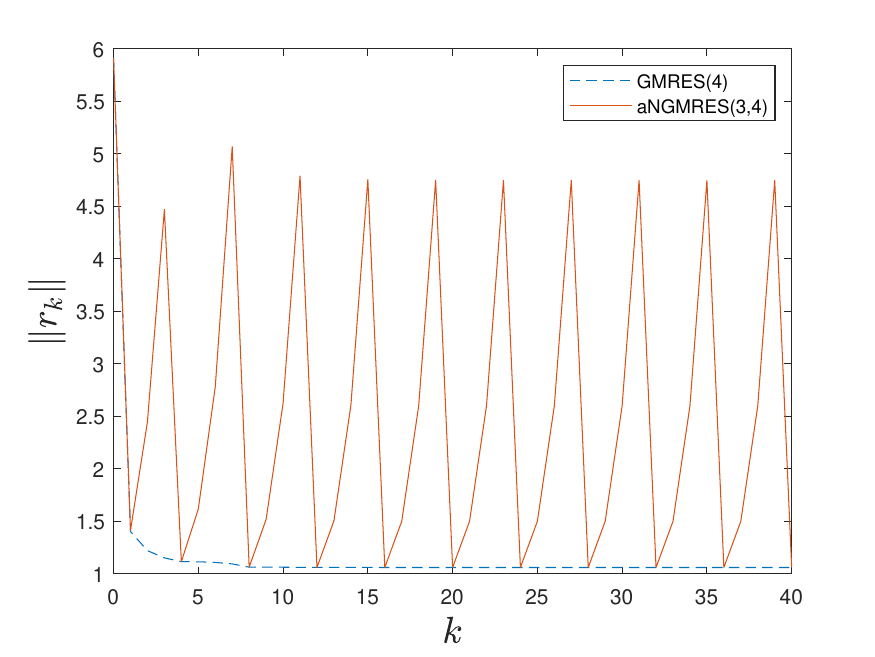}
 		\caption{Residual norm convergence history of  aNGMRES(3,4) vs. GMRES(4) for Example \ref{ex:cirex}.}\label{fig:aNG=rGMRES}
 	\end{figure}
 	
 	To validate Theorem \ref{aNGMRES-GMRES}, we consider the same example with $m=\infty$ and $p=4, 5$. Since GMRES does not stagnate, according to Theorem \ref{aNGMRES-GMRES}, aNGMRES($\infty,p$) always converges to the exact solution in finite iterations. The left of Figure \ref{fig:aNG=GMRES} shows that the iterates of aNGMRES($\infty$,4) and GMRES are coincided at the $(4j)$th iteration. Moreover, since the stagnation does not occur, GMRES and aNGMRES($\infty$,4) converge to the exact solution at the 36th iteration. For the right of Figure \ref{fig:aNG=GMRES}, we consider $m=\infty$ and $p=5$. Note that $36/5$ is not an integer.  According to Theorem \ref{aNGMRES-GMRES}, aNGMRES($\infty$,5) will converge to the exact solution at $40$th iteration, which is confirmed by our numerical result.
 	
 	\begin{figure}[h]
 		\centering
 		\includegraphics[width=7.5cm]{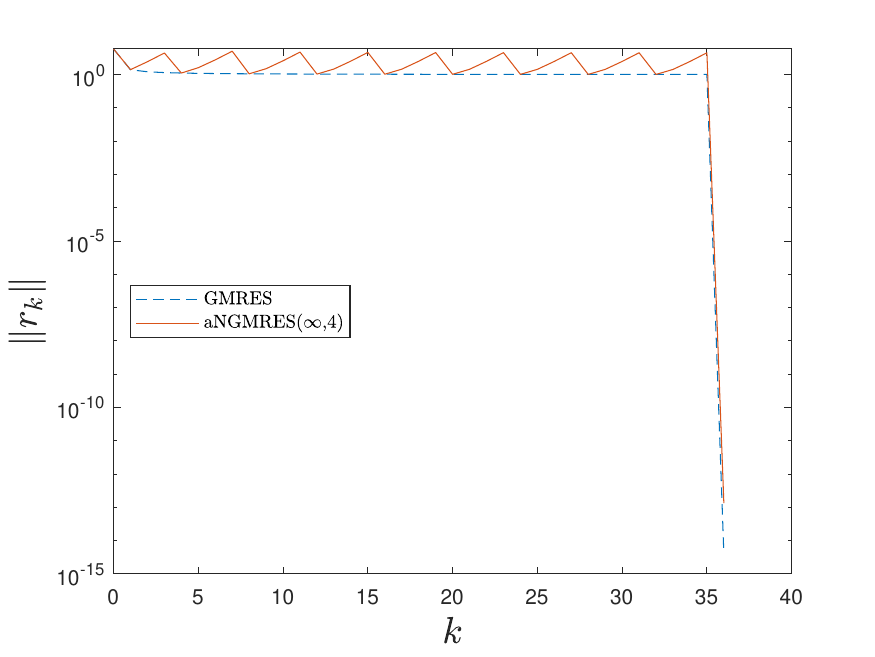}
 		\includegraphics[width=7.5cm]{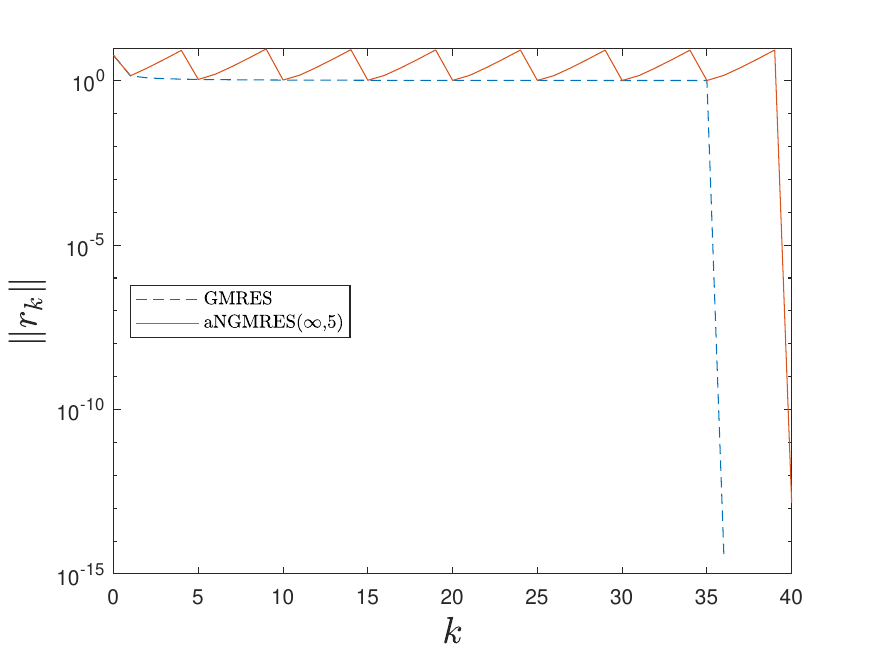}
 		\caption{Residual norm convergence history for Example \ref{ex:cirex}. Left: aNGMRES($\infty$,4) vs. GMRES. Right: aNGMRES($\infty$,5) vs. GMRES.}\label{fig:aNG=GMRES}
 	\end{figure}

 	When GMRES stagnation occurs, the convergence behavior of aNGMRES($\infty$,4) gets complicated. We use the following example to illustrate this. 
 	\begin{example}\label{ex:blockCirex}
 		Consider $\mathcal{A}x=b$ \cite{lupo2019convergence}, where $\mathcal{A}$ is block diagonal with $q$ blocks. Each block has the same structure as $A$ defined in \eqref{eq:cirmatrix}. The $k$th block has dimension $\ell k$, where $\ell \in \mathbb{N}$.
 		
 		Let us consider $\ell=3$ and $q=5$. Then, $\mathcal{A} \in \mathbb{R}^{45\times 45} (3+6+9+12+15=45)$. Note that $\mathcal{A}$ is invertible. We consider a zero initial guess, and construct $b$ by setting to one all the entries whose index coincides with the row index of matrix $\mathcal{A}$ where a new block starts, i.e., the nonzero components of vector $b$ are at indexes of 1, 4, 10, 19, 31.  $\nu(A,r_0)=30$. In this example, the full GMRES will stagnate for $\ell$ consecutive iterations. 
 	\end{example}	
 	In Figure \ref{fig:aNGs-mp}, we consider the performance of aNGMRES(2,3) and GMRES(3).  We see that GMRES(3) stagnates for three consecutive iterations, and $\|r_{3j}\|_2$ of aNGMRES(2,3) is the same as $\|r_{3j}^{G(3)}\|_2$ (i.e., $x_{3j}=x_{3j}^{G(3)}$).

 	\begin{figure}[h]
 		\centering
 		\includegraphics[width=8.5cm]{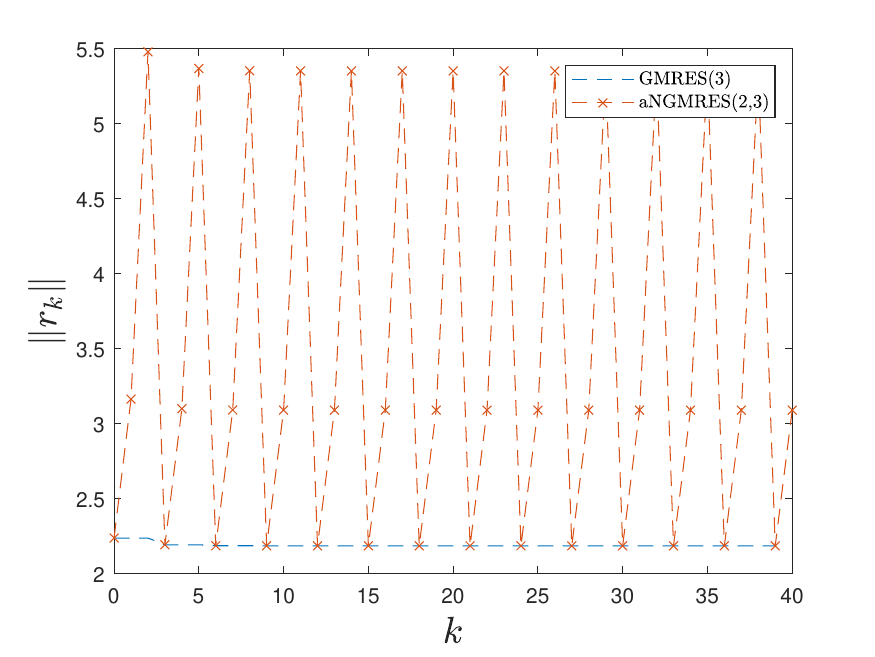}
 		\caption{Residual norm convergence history for Example \ref{ex:blockCirex}: aNGMES(2,3) vs. GMRES(3).}\label{fig:aNGs-mp}
 	\end{figure}
 	
 	In Figure \ref{fig:aNG-inf-p1}, we consider the performance of full GMRES and aNGMRES($\infty$,$p$) with $p=1, 2, 3,4$. Note that aNGMRES($\infty$,1) is NGMRES($\infty$). It can be seen that GMRES stagnates for three consecutive iterations, while aNGMRES($\infty$,1) stagnates forever, i.e., $u_k=u_0$ for all $k$. For  aNGMRES($\infty$,2), we see the norm of the residual sequence forms an alternating periodic sequence with a period of 2. However, we see that $\|r_{3j}\|$ of aNGMRES($\infty$,3) is the same as $\|r_{3j}^G\|$ (i.e., $x_{3j}=x_{3j}^G$), and aNGMRES($\infty$,3) converges to the exact solution at the 30th iteration, which is true for GMRES.
 	Moreover, aNGMRES($\infty$, 4) converges at the 40th iteration.
 	
 	From this example, we see that when GMRES stagnation occurs, the performance of aNGMRES($\infty$,$p$) becomes complicated. For NGMRES($\infty$) (i.e., aNGMRES($\infty$,1)), it can stagnate forever, but aNGMRES($\infty$,$p$) with proper choice of $p$ can converge to the exact solution within finite iterations. This enhances our understanding of NGMRES method.
 	
 	\begin{figure}[h]
 		\centering
 		\includegraphics[width=7.5cm]{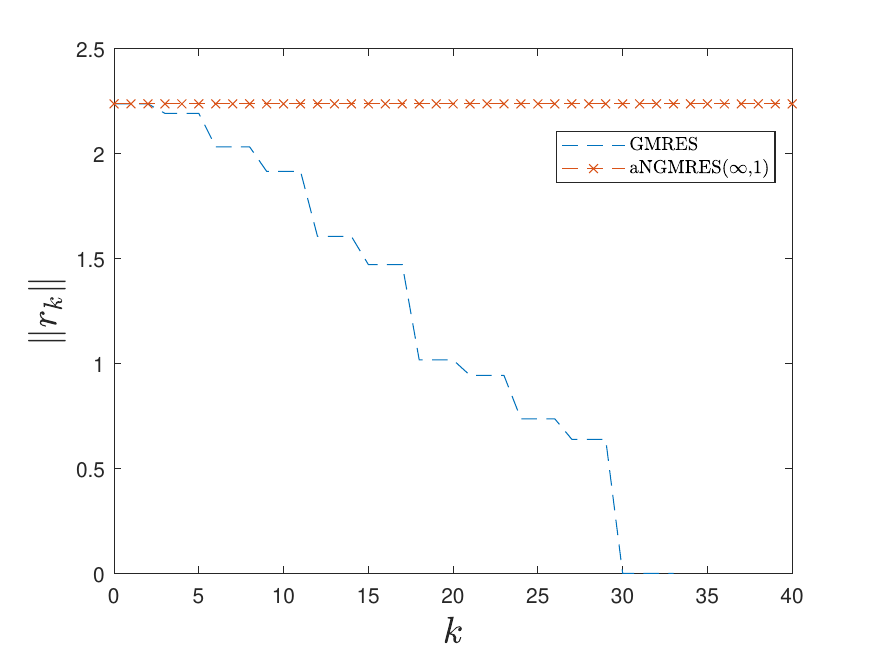}
 		\includegraphics[width=7.5cm]{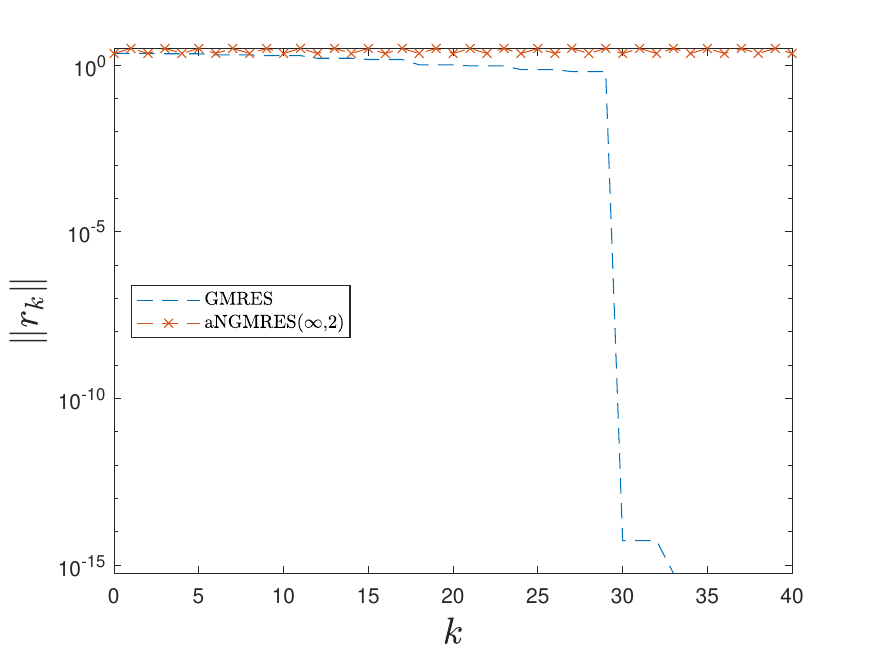}
 		\includegraphics[width=7.5cm]{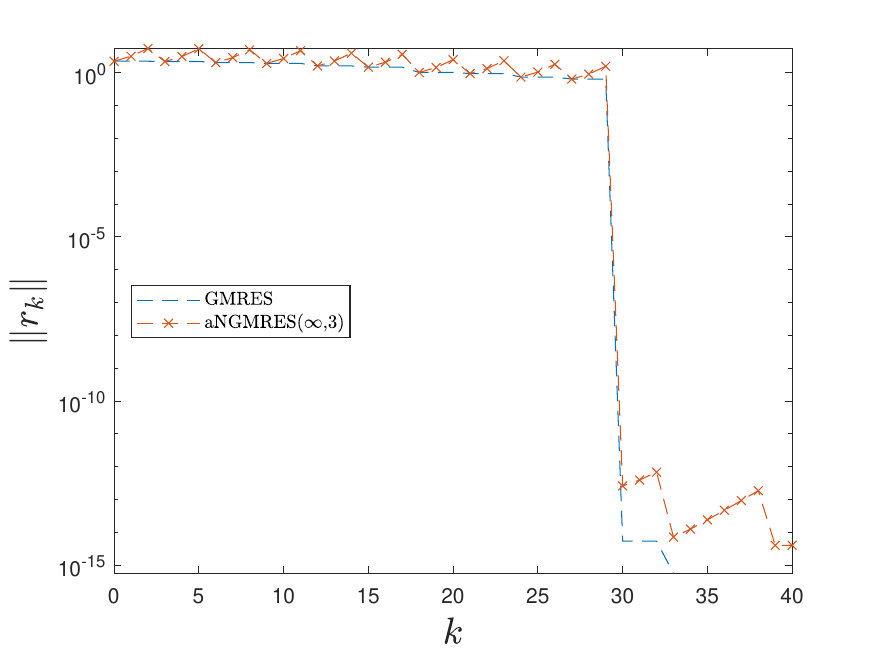}
 		\includegraphics[width=7.5cm]{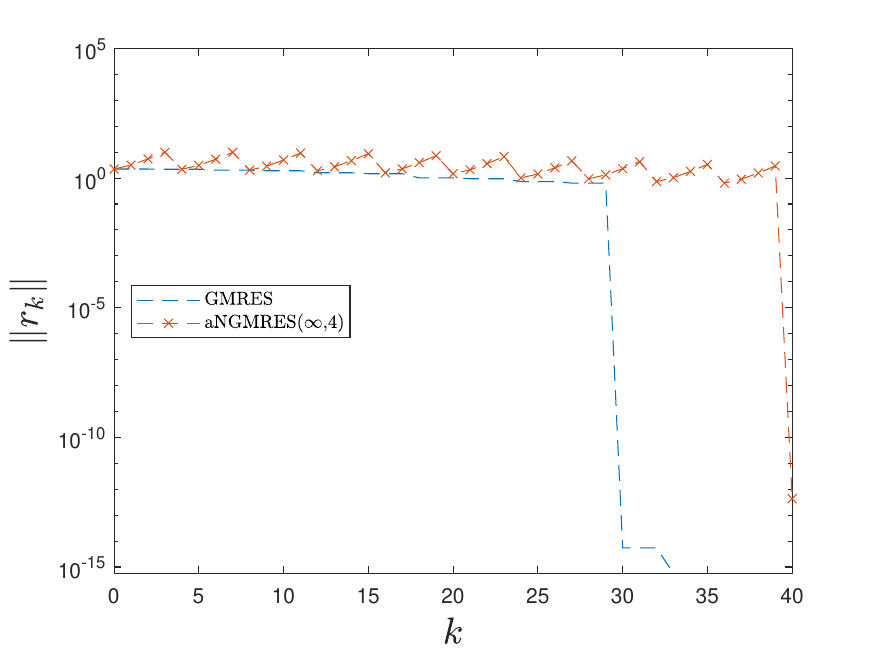}
 		\caption{Residual norm convergence history for Example \ref{ex:blockCirex}:  aNGMRES($\infty$,$p$) with $p=1, 2, 3,4$ vs. GMRES.}\label{fig:aNG-inf-p1}
 	\end{figure}
 	
 	\begin{example}\label{ex:Lap}
 		We consider a 2D Laplacian with finite difference scheme and construct the coefficient matrix by using the MATLAB command $A=delsq(numgrid('S',n+2))$ with $n=64$. We take $Au=b$ with $b$ all the ones and a random initial guess. 
 	\end{example}
 	In this example, the fixed-point iteration $q(u_{k+1})=Mu_k+b$ with $M=I-A$ does not converge.  To validate Theorem \ref{aNGMRES-GMRES}, we compare  aNGMRES($\infty$,$3$) with GMRES in Figure \ref{fig:aNG-Lap}. We see that the subsequence $\{u_{3j}\}$ of the aNGMRES($\infty$,$3$) method coincides with the sequence $\{u_{3j}^G\}$ of GMRES. In general, aNGMRES($\infty$,$p$) can be used as an alternative to GMRES.
 	
 	\begin{figure}[h]
 		\centering
 		\includegraphics[width=8cm]{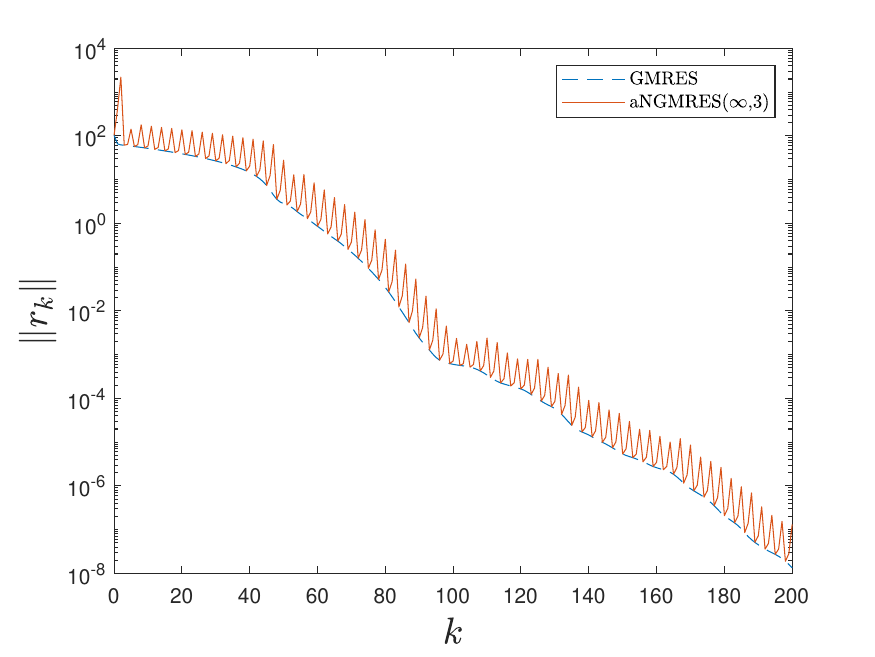}
 		\caption{Residual norm convergence history for Example \ref{ex:Lap}:  aNGMRES($\infty$,$3$) vs. GMRES.}\label{fig:aNG-Lap}
 	\end{figure}

 	\subsection{Applications}
 		Although this work primarily emphasizes the theoretical and practical aspects of aNGMRES in conjunction with GMRES, we now present a numerical study demonstrating the application of aNGMRES to a nonlinear PDE boundary-value problem.
 		
 		\begin{example}\label{ex:Bratu}
 			We consider the following 2D Bratu problem \cite{walker2011anderson} 
 			\begin{align*}
 				\Delta u + \lambda e^u =0, &  \quad \text{in}\quad \Omega=(0,1)\times (0,1),\\
 				u =0, & \quad \text{on}\quad \partial \Omega.
 			\end{align*}
 		\end{example}
 		In this experiment, we used a centered-difference discretization on an $N\times N$ grid ($h=1/(N+1)$). This results in a system of the form
 		\begin{equation}
 			Au=\lambda e^u:=f(u),
 		\end{equation}
 		where $A$ is the discrete  negative Laplacian matrix. We define $g(u)= Au-f(u)$, and look for solution of $g(u)=0$.
 		The fixed-point iteration that we consider is  
 		\begin{equation}\label{eq:Picard}
 			q(u_{k})=u_k-P^{-1}(Au_k-f(u_k)).
 		\end{equation}
 		where $P=I$ or $P={\rm diag}(A)$.  We take $\lambda=6$ in the Bratu problem and use the zero initial approximation.
 		
 		Our experiments were conducted in MATLAB. Since MATLAB timings may not accurately reflect performance in other computational environments and are not fully optimized, we have chosen to focus on reporting iteration counts rather than timing. Implementing the alternating NGMRES method in more advanced environments, such as C or Python, and comparing its performance with other approaches represents a valuable direction for future research. We perform a number of numerical tests by varying the window size $m$, periodicity $p$, and mesh size $h$. The stopping criterion is $\|r_k\|\leq 10^{-8}$.
 		
 		When $P=I$, the fixed-point iteration defined by \eqref{eq:Picard} corresponds to the classical Picard iteration and it diverges in this case. In our experiments, both AA and aNGMRES also exhibit a divergent behavior under this setting. Interestingly, NGMRES($m$) converges for certain values of $m$, such as 2, 6, 10, 20. In Figure \ref{fig:Bratu-NG}, we report the performance of NGMRES(2), NGMRES(10), NGMRES(15), and NGMRES(20). Among these,  NGMRES(15) demonstrates the best performance, followed by  NGMRES(10) and NGMRES(2).  Surprisingly, NGMRES(20) performs nearly the same as NGMRES(2). It may suggest that excessively large window sizes $m$ may not yield additional benefits and could even hinder convergence efficiency.  
 		
 		\begin{figure}[h]
 			\centering
 			\includegraphics[width=8cm]{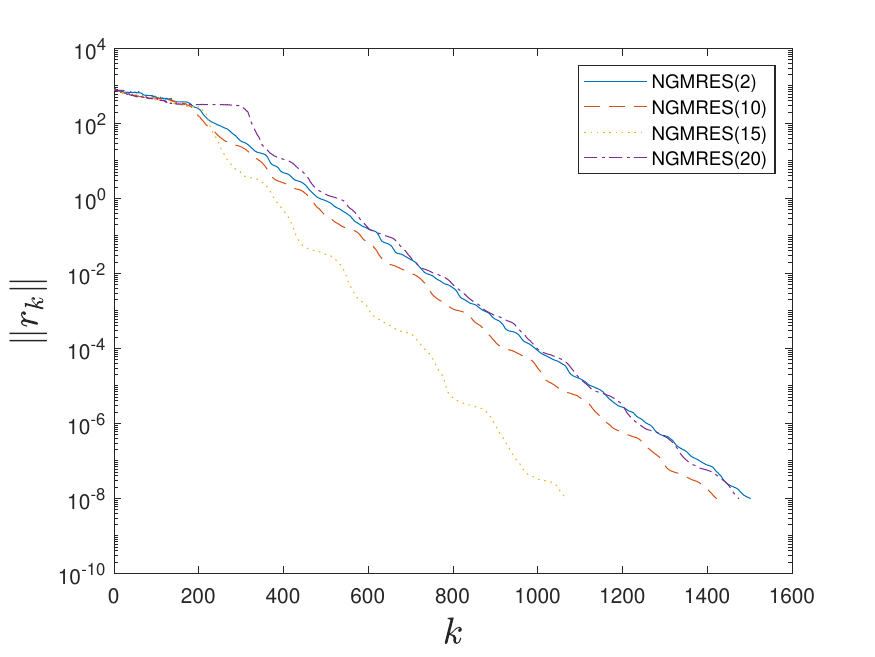}
 			\caption{Residual norm convergence history for Example \ref{ex:Bratu} for NGMRES($m$) with $h=1/129$ and $P=I$, and $m=2, 10, 15, 20$.}\label{fig:Bratu-NG}
 		\end{figure}
 		
 		When $P={\rm diag}(A)$, the fixed-point iteration defined by \eqref{eq:Picard} exhibits very slow convergence. To investigate this behavior, we focus on this iteration and perform a series of tests using AA($m$), NGMRES($m$), aAA($m,p$) and aNGMRES($m,p$), systematically varying the parameters $m$ and $p$. 
 		
 		In Figure \ref{fig:Bratu-NG-AAmp22-52}, we report the residual norm convergence history of AA($m$), NGMRES($m$), aAA($m,p$) and aNGMRES($m,p$) for $(m,p)=(2,2)$ and $(m,p)=(5,2)$. The left panel of Figure \ref{fig:Bratu-NG-AAmp22-52} shows that NGMRES(2) and aNGMRES(2,2) outperform both AA(2) and aAA(2,2), and AA and aAA stagnate. Among them, NGMRES(2) requires fewer iterations than aNGMRES(2,2) to reach the stopping criterion. 
 		
 		In contrast,  the right pane of Figure \ref{fig:Bratu-NG-AAmp22-52}, corresponding to $(m,p)=(5,2)$, shows that NGMRES(5) and aNGMRES(5,2) exhibit nearly identical convergence behavior, both outperforming AA(5) and aAA(5,2). In this case, we expect aNGMRES to achieve a lower runtime than NGMRES. Additionally, we see aAA(5,2) converges much faster than AA(2).
 		
 		In Figure \ref{fig:Bratu-NG-AAmp155}, we consider $(m,p)=(15,2)$ and $(m,p=(15,5)$. From the left panel of Figure \ref{fig:Bratu-NG-AAmp155}, we observe that  NGMRES(15) and aNGMRES(15,2) exhibit a similar convergence behavior and outperform both AA(15) and aAA(15,2). In the right panel of Figure \ref{fig:Bratu-NG-AAmp155},  NGMRES(15), aNGMRES(15,5) and aAA(15,5) show comparable convergence histories, all outperforming AA(15).
 		
 		In Figure \ref{fig:Bratu-NG-AAmp20}, we examine the performance of the methods with parameter settings $(m,p)=(20,2)$ and $(m,p=(20,5)$. Notably,  AA(20) performs worse than AA(10), as shown in Figure \ref{fig:Bratu-NG-AAmp155}. GMRES(20), aNGMRES(20,2), and aAA(20,2) exhibit similar performance.  However, aAA(20,5) shows a slight improvement over both GMRES(20), aNGMRES(20,5).   NGMRES(20) and NGMRES(15) demonstrate comparable performance, as illustrated in Figure \ref{fig:Bratu-NG-AAmp155}.
 		
 		\begin{figure}[h]
 			\centering
 			\includegraphics[width=7.5cm]{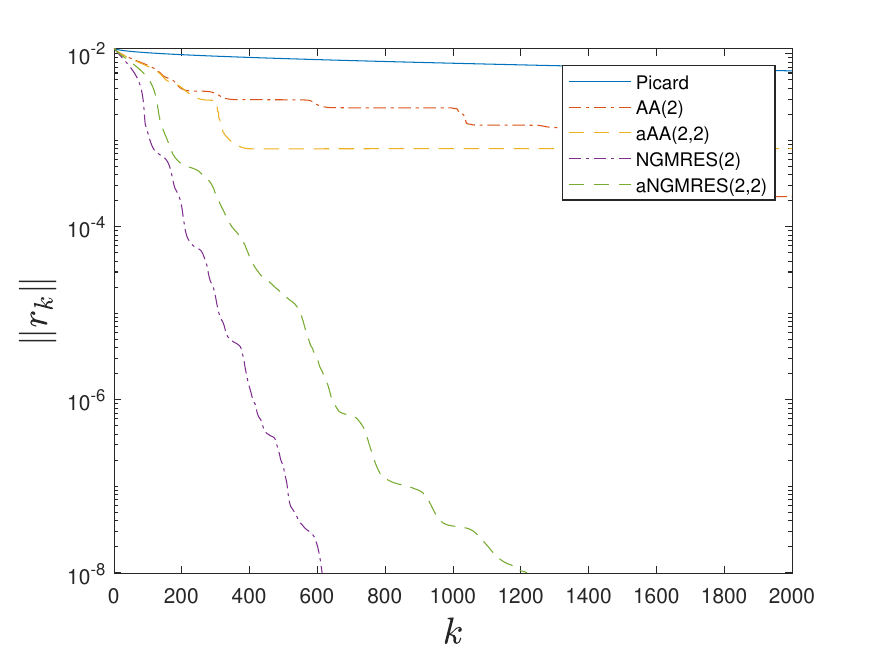}
 			\includegraphics[width=7.5cm]{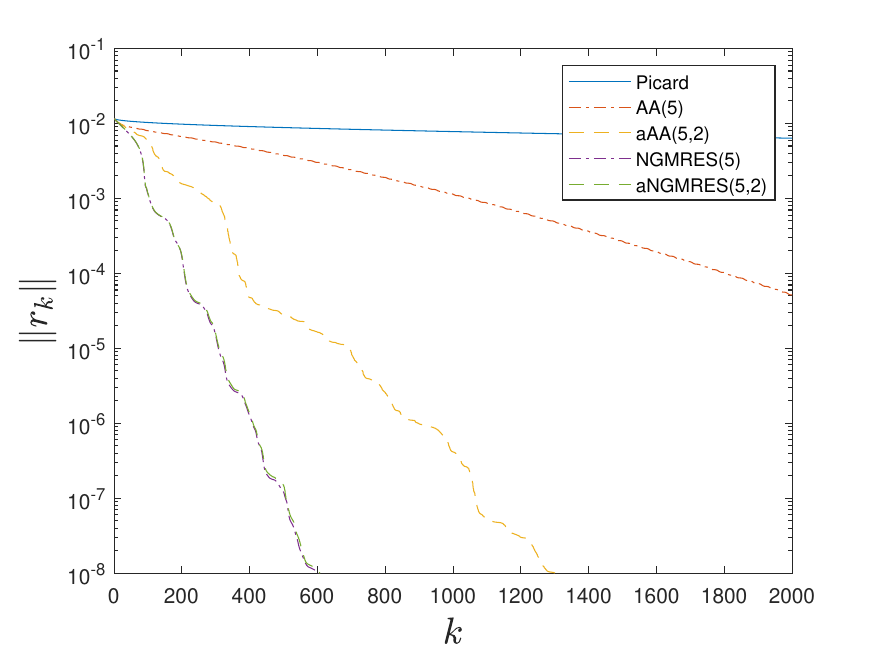}
 			\caption{Residual norm convergence history for Example \ref{ex:Bratu} for AA($m$), NGMRES($m$), aAA($m,p$) and aNGMRES($m,p$) with $h=1/129$ and $P={\rm diag}(A)$. Left: $(m,p)=(2,2)$. Right: $(m,p)=(5,2)$. }\label{fig:Bratu-NG-AAmp22-52}
 		\end{figure}
 		
 		\begin{figure}[h]
 			\centering
 			\includegraphics[width=7.5cm]{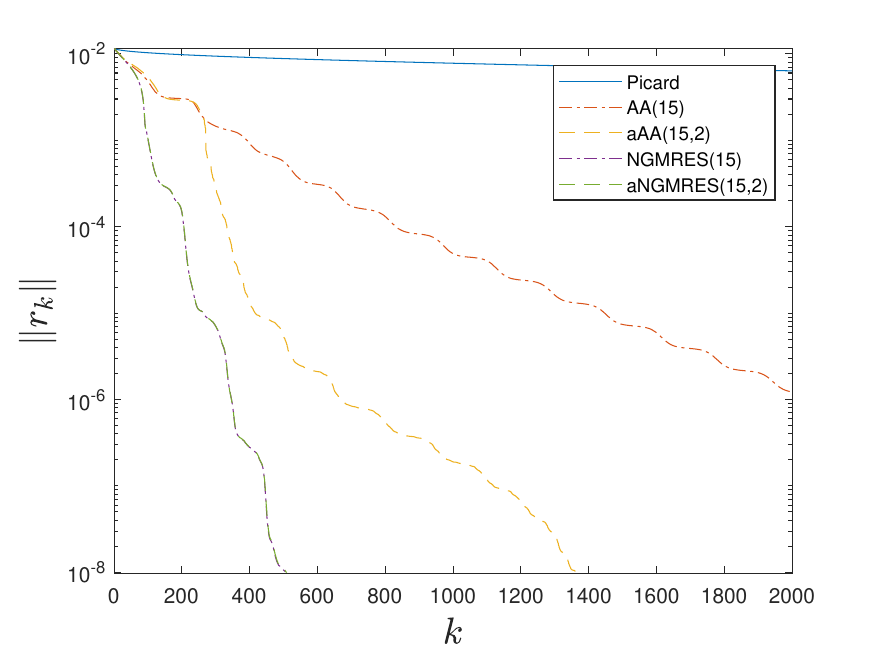}
 			\includegraphics[width=7.5cm]{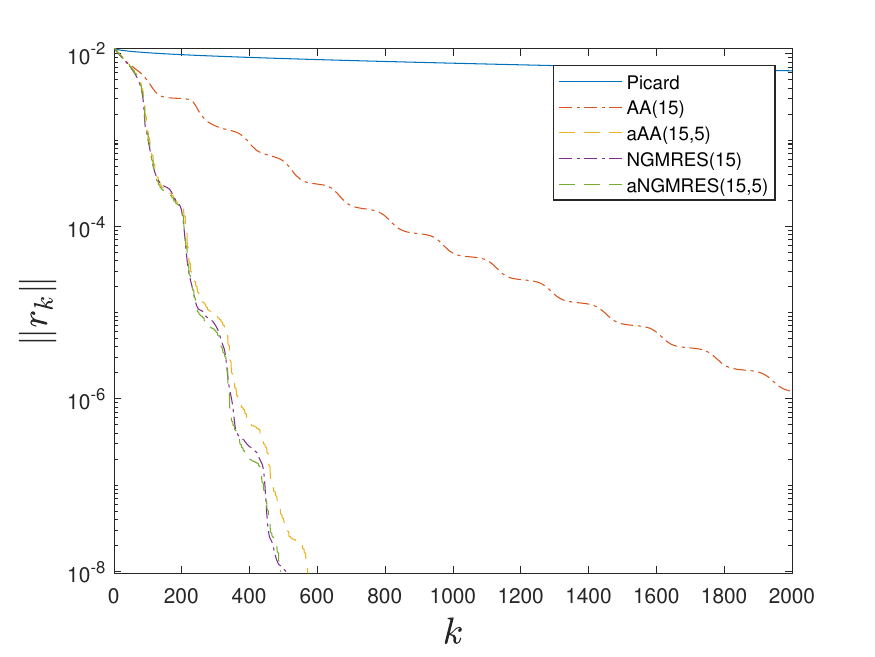}
 			\caption{Residual norm convergence history for Example \ref{ex:Bratu} for AA($m$), NGMRES($m$), aAA($m,p$) and aNGMRES($m,p$) with $h=1/129$ and $P={\rm diag}(A)$. Left: $(m,p)=(15,2)$. Right: $(m,p)=(15,5)$. }\label{fig:Bratu-NG-AAmp155}
 		\end{figure}

 		\begin{figure}[h]
 			\centering
 			\includegraphics[width=7.5cm]{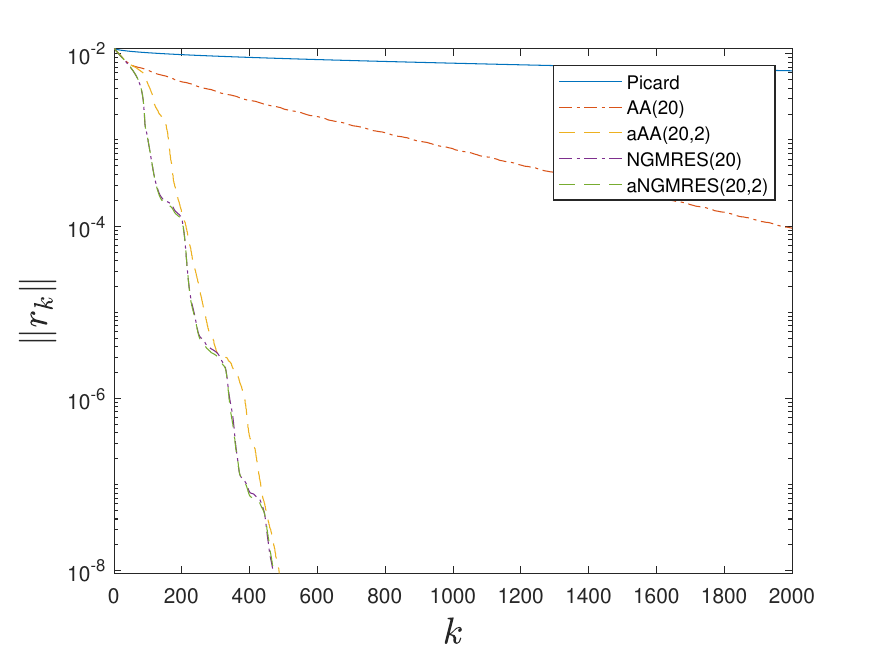}
 			\includegraphics[width=7.5cm]{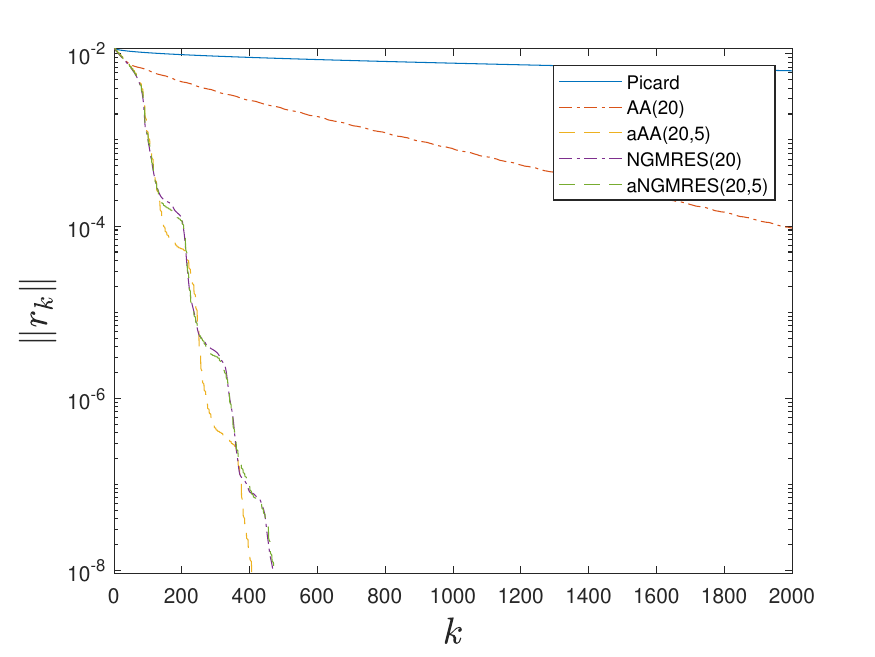}
 			\caption{Residual norm convergence history for Example \ref{ex:Bratu} for AA($m$), NGMRES($m$), aAA($m,p$) and aNGMRES($m,p$) with $h=1/129$ and $P={\rm diag}(A)$. Left: $(m,p)=(20,2)$. Right: $(m,p)=(20,5)$. }\label{fig:Bratu-NG-AAmp20}
 		\end{figure}

 		In Figure \ref{fig:Bratu-NG-AA-h}, we examine the performance of aAA(20,2) and aNGMRES(20,2) by varying the mesh size $h=1/(N+1)$. The tests are labeled as aAA-N and aNGMRES-N, corresponding to $N=32, 64, 128, 256$.  The number of iterations required to meet the stopping criterion for aAA($20,2$) are 114, 261, 486, and 1268, while for aNGMRES($20,2$) they are 80, 196, 470, and 1074, respectively.  The computational cost per iteration does not make a big difference between the two methods.  We also  report the iteration counts for aAA($20,5$) and aNGMRES($20,5$). The iteration counts for aAA($20,5$) are 70, 175, 405, and 1135, and for aNGMRES($20,5$) are 75, 195, 475, 1085.

 		\begin{figure}[h]
 			\centering
 			\includegraphics[width=7.5cm]{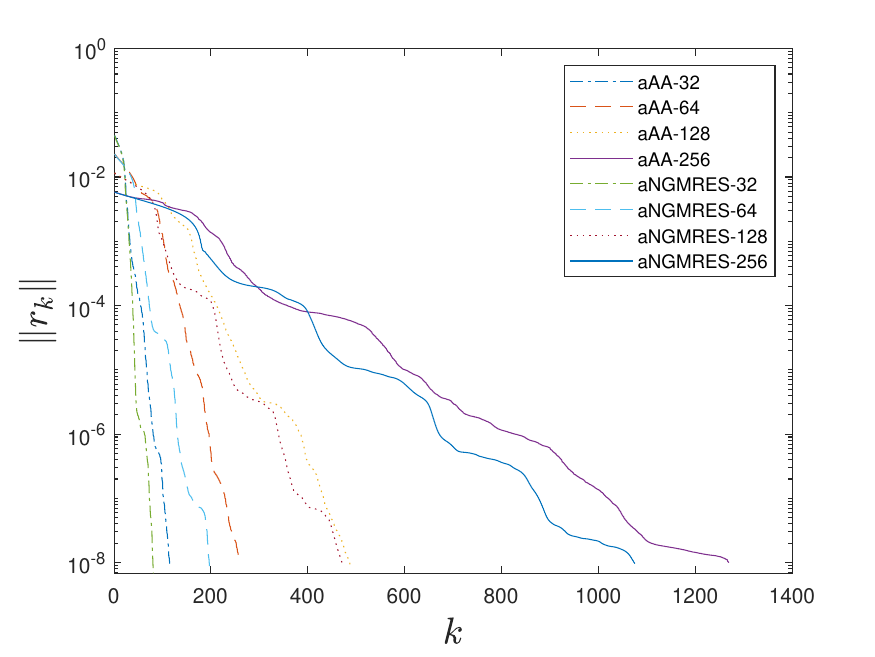}
 			\includegraphics[width=7.5cm]{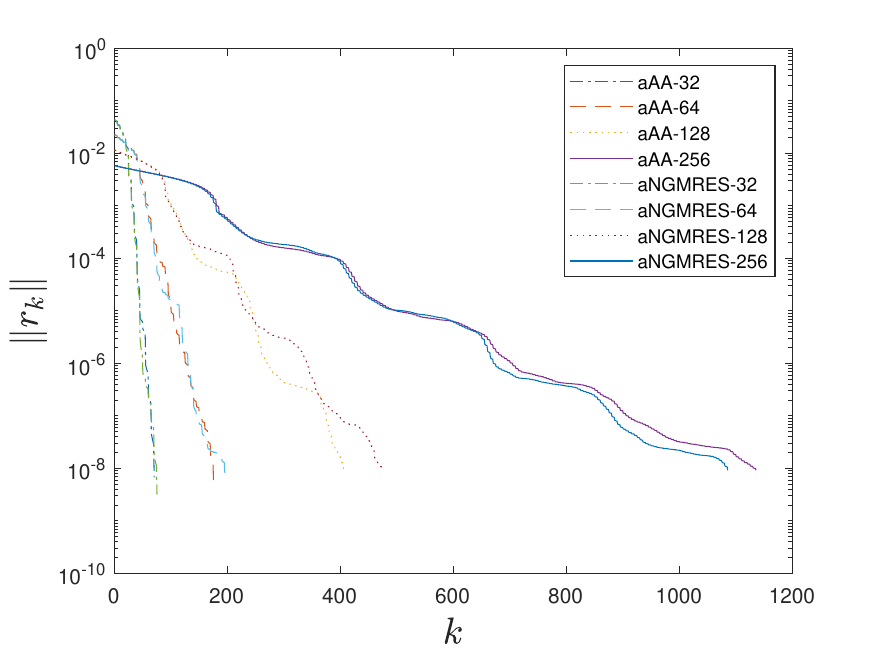}
 			\caption{Residual norm convergence history for Example \ref{ex:Bratu} for aAA and aNGMRES  with $P={\rm diag}(A)$, varying $N$. Left: aAA($20,2$) and aNGMRES($20,2$). Right: aAA($20,5$) and aNGMRES($20,5$).  }\label{fig:Bratu-NG-AA-h}
 		\end{figure}
 		
 		In summary, this example demonstrates that NGMRES can converge even for noncontractive fixed-point iterations. The adaptive Anderson aAA significantly outperforms standard AA. With appropriate choices of $m$ and $p$, NGMRES($m$), aAA($m,p$), NGMRES($m$), and aNGMRES($m,p)$ exhibit comparable performance. However, for large values of $m$, both AA($m$) and NGMRES($m$) ($m\geq 15$) show limited improvement.
 		
 		This is a preliminary case study. In future work, we will explore the application of aNGMRES to more complex systems, particularly nonlinear problems that pose significant computational challenges. A detailed comparison with Anderson Acceleration   will be conducted to assess its performance relative to NGMRES or aNGMRES.

 	\section{Conclusion}\label{sec:con}
 	In this work, we analyze one-step NGMRES($m$) applied to linear systems, where the $m$ iterates are obtained by Richardson iteration, then one iteration of NGMRES($m$) is applied to these iterates. Based on this, we propose alternating NGMRES with depth $m$ and periodicity $p$, denoted as aNGMRES($m,p$). The proposed method avoids solving the least-squares problem at each iteration, which can save computational work. We establish the relationship between aNGMRES($m,p$), full GMRES, and restarted GMRES. aNGMRES($\infty,p$) can be regarded as an alternative to GMRES for solving linear systems. For finite $m$, the iterates of aNGMRES($m,m+1$) and restarted GMRES (GMRES($m+1$)) are identical at the iteration index $jp$.  These findings increase our understanding of NGMRES. Finally, we present a convergence analysis of aNGMRES applied to linear systems, using Richardson iteration as the underlying fixed-point scheme. Numerical experiments validate our theoretical findings.
 	
 	Although the fixed-point iteration considered in this work is the Richardson iteration, the proposal aNGMRES can be applied to other fixed-point iterations, including nonlinear ones. In the future, we will explore the application of aNGMRES to more challenging linear and nonlinear systems, conduct convergence analysis for nonlinear case, and develop efficient solvers for the least-squares problems involved in aNGMRES.

\bibliographystyle{plain}
\bibliography{aNGMRESBib}

\begin{thebibliography}{10}

\bibitem{anderson1965iterative}
Donald~G Anderson.
\newblock Iterative procedures for nonlinear integral equations.
\newblock {\em Journal of the ACM (JACM)}, 12(4):547--560, 1965.

\bibitem{anderson2019comments}
Donald~GM Anderson.
\newblock Comments on “{A}nderson acceleration, mixing and extrapolation”.
\newblock {\em Numerical Algorithms}, 80:135--234, 2019.

\bibitem{chen2022composite}
Kewang Chen and Cornelis Vuik.
\newblock Composite {A}nderson acceleration method with two window sizes and
  optimized damping.
\newblock {\em International Journal for Numerical Methods in Engineering},
  123(23):5964--5985, 2022.

\bibitem{sterck2012nonlinear}
Hans De~Sterck.
\newblock A nonlinear {GMRES} optimization algorithm for canonical tensor
  decomposition.
\newblock {\em SIAM Journal on Scientific Computing}, 34(3):A1351--A1379, 2012.

\bibitem{sterck2013steepest}
Hans De~Sterck.
\newblock Steepest descent preconditioning for nonlinear {GMRES} optimization.
\newblock {\em Numerical Linear Algebra with Applications}, 20(3):453--471,
  2013.

\bibitem{sterck2021asymptotic}
Hans De~Sterck and Yunhui He.
\newblock On the asymptotic linear convergence speed of {A}nderson
  acceleration, {N}esterov acceleration, and nonlinear {GMRES}.
\newblock {\em SIAM Journal on Scientific Computing}, 43(5):S21--S46, 2021.

\bibitem{d2021acceleration}
Alexandre d’Aspremont, Damien Scieur, Adrien Taylor, et~al.
\newblock Acceleration methods.
\newblock {\em Foundations and Trends{\textregistered} in Optimization},
  5(1-2):1--245, 2021.

\bibitem{evans2020proof}
Claire Evans, Sara Pollock, Leo~G Rebholz, and Mengying Xiao.
\newblock A proof that {A}nderson acceleration improves the convergence rate in
  linearly converging fixed-point methods (but not in those converging
  quadratically).
\newblock {\em SIAM Journal on Numerical Analysis}, 58(1):788--810, 2020.

\bibitem{fang2009two}
Haw-ren Fang and Yousef Saad.
\newblock Two classes of multisecant methods for nonlinear acceleration.
\newblock {\em Numerical linear algebra with applications}, 16(3):197--221,
  2009.

\bibitem{GreifHe25NGMRES}
Chen Greif and Yunhui He.
\newblock Convergence properties of nonlinear {GMRES} applied to linear
  systems.
\newblock {\em SIAM Journal on Matrix Analysis and Applications}, 2025.

\bibitem{He25NGMRES}
Yunhui He.
\newblock Convergence analysis for nonlinear {GMRES}.
\newblock {\em arXiv preprint arXiv:2501.09634}, 2025.

\bibitem{He25worstcase}
Yunhui He.
\newblock The worst-case root-convergence factor of {GMRES}(1).
\newblock {\em arXiv preprint arXiv:2501.10248}, 2025.

\bibitem{lockhart2022performance}
Shelby Lockhart, David~J Gardner, Carol~S Woodward, Stephen Thomas, and Luke~N
  Olson.
\newblock Performance of low synchronization orthogonalization methods in
  {A}derson accelerated fixed point solvers.
\newblock In {\em Proceedings of the 2022 SIAM Conference on Parallel
  Processing for Scientific Computing}, pages 49--59. SIAM, 2022.

\bibitem{loffeld2016considerations}
John Loffeld and Carol~S Woodward.
\newblock Considerations on the implementation and use of {A}nderson
  acceleration on distributed memory and {GPU}-based parallel computers.
\newblock In {\em Advances in the Mathematical Sciences: Research from the 2015
  Association for Women in Mathematics Symposium}, pages 417--436. Springer,
  2016.

\bibitem{lupo2019convergence}
Massimiliano Lupo~Pasini.
\newblock Convergence analysis of {A}nderson-type acceleration of
  {R}ichardson's iteration.
\newblock {\em Numerical Linear Algebra with Applications}, 26(4):e2241, 2019.

\bibitem{nesterov2013introductory}
Yurii Nesterov.
\newblock {\em Introductory lectures on convex optimization: {A} basic course},
  volume~87.
\newblock Springer Science \& Business Media, 2013.

\bibitem{oosterlee2000krylov}
Cornelis~W Oosterlee and Takumi Washio.
\newblock Krylov subspace acceleration of nonlinear multigrid with application
  to recirculating flows.
\newblock {\em SIAM Journal on Scientific Computing}, 21(5):1670--1690, 2000.

\bibitem{pollock2019anderson}
Sara Pollock, Leo~G Rebholz, and Mengying Xiao.
\newblock Anderson-accelerated convergence of {P}icard iterations for
  incompressible {N}avier--{S}tokes equations.
\newblock {\em SIAM Journal on Numerical Analysis}, 57(2):615--637, 2019.

\bibitem{potra2013characterization}
Florian~A Potra and Hans Engler.
\newblock A characterization of the behavior of the {A}nderson acceleration on
  linear problems.
\newblock {\em Linear Algebra and its Applications}, 438(3):1002--1011, 2013.

\bibitem{pratapa2016anderson}
Phanisri~P Pratapa, Phanish Suryanarayana, and John~E Pask.
\newblock Anderson acceleration of the {J}acobi iterative method: {A}n
  efficient alternative to {K}rylov methods for large, sparse linear systems.
\newblock {\em Journal of Computational Physics}, 306:43--54, 2016.

\bibitem{saad2003iterative}
Yousef Saad.
\newblock {\em Iterative methods for sparse linear systems}.
\newblock SIAM, 2003.

\bibitem{saad1986gmres}
Yousef Saad and Martin~H Schultz.
\newblock {GMRES}: A generalized minimal residual algorithm for solving
  nonsymmetric linear systems.
\newblock {\em SIAM Journal on Scientific and Statistical Computing},
  7(3):856--869, 1986.

\bibitem{suryanarayana2019alternating}
Phanish Suryanarayana, Phanisri~P Pratapa, and John~E Pask.
\newblock Alternating {A}nderson--{R}ichardson method: {A}n efficient
  alternative to preconditioned {K}rylov methods for large, sparse linear
  systems.
\newblock {\em Computer Physics Communications}, 234:278--285, 2019.

\bibitem{walker2011anderson}
Homer~F Walker and Peng Ni.
\newblock Anderson acceleration for fixed-point iterations.
\newblock {\em SIAM Journal on Numerical Analysis}, 49(4):1715--1735, 2011.

\bibitem{washio1997krylov}
Takumi Washio and Cornelis~W Oosterlee.
\newblock Krylov subspace acceleration for nonlinear multigrid schemes.
\newblock {\em Electron. Trans. Numer. Anal}, 6(271-290):3, 1997.

\bibitem{yang2022anderson}
Yunan Yang, Alex Townsend, and Daniel Appel{\"o}.
\newblock Anderson acceleration based on the ${H}^{-s}$ {S}obolev norm for
  contractive and noncontractive fixed-point operators.
\newblock {\em Journal of Computational and Applied Mathematics}, 403:113844,
  2022.

\end{thebibliography}
\end{document}